\documentclass[10pt]{amsart}

\usepackage{a4wide}
\usepackage{times}
\usepackage{amsmath, amsthm, amssymb}
\usepackage{amsfonts, stmaryrd, mathrsfs}
\usepackage{amscd}
\usepackage{enumerate}
\usepackage{hyperref}
\usepackage{cleveref}
\usepackage{tikz-cd}

\theoremstyle{plain}
\newtheorem{theorem}{Theorem}[section]
\newtheorem{prop}[theorem]{Proposition}
\newtheorem{lemma}[theorem]{Lemma}
\newtheorem{cor}[theorem]{Corollary}
\newtheorem{conj}[theorem]{Conjecture}

\theoremstyle{definition}
\newtheorem{definition}[theorem]{Definition}

\theoremstyle{remark}
\newtheorem{remark}[theorem]{Remark}

\newcommand{\A}{\mathbb{A}}
\newcommand{\C}{\mathbb{C}}

\renewcommand{\P}{\mathbb{P}}
\newcommand{\Q}{\mathbb{Q}}

\newcommand{\Z}{\mathbb{Z}}
\newcommand{\G}{\mathbb{G}}

\newcommand{\id}{\mathrm{id}}   
\newcommand{\codim}{\mathrm{codim}} 
\newcommand{\Bl}{\mathrm{Bl}} 
\newcommand{\spec}{\mathrm{Spec}} 
\newcommand{\Het}{\mathrm{H}_{\acute{e}t}} 
\newcommand{\Hrnr}[1]{\mathrm{H}_{{#1},\mathrm{nr}}} 
\newcommand{\mrH}{\mathrm{H}} 
\newcommand{\e}{\acute{e}} 

\DeclareMathOperator{\coker}{coker}
\DeclareMathOperator{\cl}{cl}
\DeclareMathOperator{\im}{im}

\DeclareMathOperator{\CH}{CH}

\DeclareMathOperator{\Hom}{Hom}

\DeclareMathOperator{\Pic}{Pic}

\newcommand{\Oc}{{\mathcal{O}}} 

\newcommand{\isomto}{\;\tilde{\rightarrow}\;}

\date{}
\title{On the failure of the integral Hodge/Tate conjecture for products with projective hypersurfaces}
\begin{document}
\author{Kees Kok}
\address{\parbox{0.9\textwidth}{KdV Institute for Mathematics, University of Amsterdam, Netherlands}}
\email{k.kok@uva.nl}
\begin{abstract}
In this paper we show the failure of the integral Hodge/Tate conjecture for the product of an Enriques surface with a smooth odd-dimensional projective hypersurface. To do this, we use a specialization argument of Colliot-Thélène (\cite{CT}) applied to Schreieder's refined unramified cohomology (\cite{Sch1}). The results obtained in this way give an interpretation of Shen's result (\cite{S}) in terms of refined unramified cohomology. Moreover, using this interpretation, we avoid the need to work over the complex numbers so that we may conclude that Shen's result also holds over general algebraically closed fields of characteristic not 2.

\end{abstract}
\maketitle

\section{Introduction}
Let $X$ be a smooth projective complex variety and write $\cl^i_\Q\colon \CH^i(X)_\Q\to\mrH^{2i}(X,\Q(i))$ for the rational cycle class map from the Chow group to the Betti cohomology of $X$. It is known that its image is contained in the Hodge classes $\mathrm{Hdg}^{2i}(X,\Q):=\mrH^{2i}(X,\Q)\cap\mrH^{i,i}(X)$, where $\mrH^k(X,\Q)\otimes\C\cong\bigoplus_{p+q=k}\mrH^{p,q}(X)$ is the Hodge decomposition of $X$. The famous \emph{Hodge conjecture}, first introduced by W.V.D. Hodge in \cite{original-HC} and later modified by A. Grothendieck \cite{adjusted-HC}, predicts that all the Hodge classes are algebraic. The modern formulation is as follows.
\begin{conj}[Hodge Conjecture]\label{HC}
Let $X$ be a smooth projective complex variety, then $\im(\cl^i_\Q)=\mathrm{Hdg}^{2i}(X,\Q)$.
\end{conj}
The Hodge conjecture thus says that topological properties of algebraic varieties can be studied algebraically. We note that it also has a motivic nature as it implies all standard conjectures in characteristic zero \cite[Chapitre 5]{Andre}. A lot of cases are known, but in general the conjecture is wide open. The interested reader can look at \cite{known-cases} for a nice overview.

In view of Hodge's original statement, one could ask what happens if $\Q$ is replaced by $\Z$. This is known as the \emph{integral} Hodge conjecture.
\begin{conj}[Integral Hodge Conjecture]\label{IHC}
Let $X$ be a smooth projective complex variety, then all integral Hodge classes are integral linear combinations of algebraic cycles, that is $\im(\cl^i)=\mathrm{Hdg}^{2i}(X,\Z):=\mrH^{2i}(X,\Z)\cap\mrH^{i,i}(X)$.
\end{conj}
\Cref{IHC} clearly implies \Cref{HC} and trivially holds for $i=0,\dim(X)$. The so called `Lefschetz-(1,1)-theorem' due to Solomon Lefschetz \cite{Lefschetz} gives the $i=1$ case. It turns out however that \Cref{IHC} fails to be true in general.

This failure was first established by Atiyah and Hirzebruch in \cite{AH}. They showed that a certain \emph{torsion} cohomology class, which is automatically Hodge, cannot be algebraic. Kollár was the first to find a Hodge class of infinite order that is not algebraic, whereas a multiple of this class is, \cite{Kol}. Thoughout the years, many more counterexamples for the integral Hodge conjecture were found, \cite{SV}, \cite{OS}, \cite{BO}, \cite{CT}, \cite{CT-V}, \cite{S}, \cite{Diaz1}, \cite{Diaz2} to name a few relatively recent ones.

Let us explain how this paper relates to some of the aforementioned ones. In \cite{BO} Benoist and Ottem showed that certain products of a projective curve with an Enriques surface violated \Cref{IHC}. This violation was given an interpretation in terms of \emph{unramified cohomology} by Colliot-Thélène in \cite{CT}. This interpretation was based on the relation between the failure of the integral Hodge conjecture in codimension 2 and the non-vanishing of certain unramified cohomology groups, observed by Colliot-Thélène and Voisin, \cite[Théorème 3.7]{CT-V}. 

Later, Shen gave in \cite{S} an interpretation of \cite{BO} in terms of a topological obstruction to algebraicity. A careful analysis of this obstruction allowed Shen to show the existence of non-algebraic Hodge classes on a product of a very general odd-dimensional projective hypersurface of degree $\geq 3$ with an Enriques surface. Shen ended his introduction by asking whether his result also has an interpretation in terms of unramified cohomology, \cite[Remark 1.6]{S}. This paper aims to give this interpretation. In the process, we replace Shen's topological obstruction arguments by purely algebraic ones and doing so, we avoid the need to work over the complex numbers. Instead, in this paper we shall work over an algebraically closed field $k$ of any characteristic (and we shall later specify to $\mathrm{char}(k)\neq 2$). Consequently, we get results involving the \emph{Integral Tate Conjecture}, which we shall describe now.

Let $X$ be a smooth projective variety over $k$ and let $\ell\neq\mathrm{char}(k)$ be a prime. Let $k_0\subseteq k$ be any subfield over which $X$ is defined so that $\bar k_0=k$, that is, there exists a variety $X_0$ over $k_0$ so that $X_0\times_{k_0}k=X$. There is a cycle class map $\cl^i_{\Z_\ell}\colon\CH^i(X)_{\Z_\ell}\to\Het^{2i}(X,\Z_\ell(i))$ and it is known that its image is contained in $\varinjlim_{k_0\subseteq k'}\Het^{2i}(X,\Z_\ell(i))^{G_{k'}}$, where the limit runs over all finite extensions $k'$ of $k_0$ and $G_{k'}$ is the absolute Galois group of $k'$, \cite[ Chapter VI \textsection 9]{Milne} and \cite[Cycle \textsection 2]{SGA42}.

\begin{conj}[Integral Tate Conjecture as in \cite{Schoen}]\label{ITC}
Assume $k_0$ is finite over its prime field, then
\[
\im(\cl^i_{\Z_\ell})=\varinjlim_{k_0\subseteq k'}\Het^{2i}(X,\Z_\ell(i))^{G_{k'}}.
\]
\end{conj}

Our main result is the following.
\begin{theorem}[cf. \Cref{failure ITC}]\label{intro thm}
Let $\mathcal X\to\P^1$ be a Lefschetz pencil of odd-dimensional projective hypersurfaces of degree $d\geq 3$ over a field $k=\bar k$ with $\mathrm{char}(k)\neq 2$ and let $S$ be an Enriques surface over the same field $k$. Then the integral Tate conjecture fails for $X_{\bar\eta}\times S$, where $X_{\bar\eta}$ is the geometric generic fibre of $\mathcal X\to\P^1$.
\end{theorem}

As mentioned before, this result is due to an interpretation of \cite{S} in terms of unramified cohomology. More precisely, our approach relies on the theory of \emph{refined unramified cohomology} developed by Schreieder in \cite{Sch1}, which generalizes the obstruction given in \cite{CT-V}. To give a more accurate statement of our main result for the moment, we quickly recall that the refined unramified cohomology group of $X$ with coefficients in $A$ is denoted by $\Hrnr{j}^i(X,A)$ (see \Cref{ref unram cohom} for precise definitions and statements), then Schreieder showed the following.
\begin{theorem}[{\cite[Theorem 7.7]{Sch1}}]
There is an isomorphism 
\[
\frac{\Hrnr{i-2}^{2i-1}(X,\mu_{\ell^r}^{\otimes i})}{\Hrnr{i-2}^{2i-1}(X,\Z_\ell(i))}\cong\coker(\cl^i_{\Z_\ell}\colon \CH^i(X)_{\Z_\ell}\to\Het^{2i}(X,\Z_\ell(i)))[\ell^r].
\]
\end{theorem}

Then \Cref{intro thm} becomes
\begin{theorem}[More precise form of \Cref{intro thm}]
If the dimension of the hypersurfaces appearing in the Lefschetz pencil $\mathcal X$ is $2n-1$, then the natural map
\[
\frac{\Het^{2n+1}(X_{\bar\eta}\times S,\mu_2^{\otimes n+1})}{\Het^{2n+1}(X_{\bar\eta}\times S,\Z_2(n+1))}\to\frac{\Hrnr{n-1}^{2n+1}(X_{\bar\eta}\times S,\mu_2^{\otimes n+1})}{\Hrnr{n-1}^{2n+1}(X_{\bar\eta}\times S,\Z_2(n+1))}
\]
is non-zero. That is, there exists a 2-torsion class in $\Het^{2n+2}(X_{\bar\eta}\times S,\Z_2(n+1))$ which is not in the image of $\cl_{\Z_2}^{n+1}$.
\end{theorem}

Having established this, we also obtain as direct consequences similar results involving $X_{\bar\eta}\times S^i$ (cf. \Cref{gen failure ITC}) and certain variety of lines (cf. \Cref{var of lines failure}).

\subsection*{Structure of the paper} In \Cref{preliminaries} we recall some general result concerning étale cohomology and cover the construction of the refined unramified cohomology groups and the specialization map from \cite{Sch2}. In \Cref{comparing sp and cosp} we introduce the cospecialization map from \cite[Chapter I\textsection 8]{FK} and show how this map compares with the aforementioned specialization map. Then in \Cref{main section} we apply everything to certain Lefschetz pencils to deduce our application on the integral Hodge/Tate type conjectures, similar to \cite{CT}. We end with some general observations in \Cref{closing remarks}.

\section*{Acknowledgements} I am very grateful to my supervisor Dr. Mingmin Shen for introducing me to this topic and for the many helpful conversations concerning the subject. This project is carried out during the author's PhD, which is supported by the NWO vidi grant 016.Vidi.189.015.

\section{Preliminaries}\label{preliminaries}
Here we recall the theory of étale cohomology and give proofs of standard properties for reference purposes. Moreover we recall the theory of \cite{Sch1} on refined unramified cohomology and some of its properties we will use.
\subsection{Étale Cohomology}
\subsubsection{Verdier Duality}\label{verdier duality}
Let $f\colon X\to S$ be a compactifiable morphism over $k$. The following is \cite[Chapitre XVIII, Théorème 3.1.4]{SGA4}.
\begin{theorem}\label{VD}
Let $\mathrm{char}(k)\nmid n$, then the functor $Rf_!\colon D(X,\Z/n\Z)\to D(S,\Z/n\Z)$ has a partial right adjoint $Rf^!\colon D^+(S,\Z/n\Z)\to D^+(X,\Z/n\Z)$. 

If we moreover assume $f$ to be smooth, then we can define $Rf^!$ as in \cite[\textsection 4.4]{Verdier}.
\end{theorem}

Consider the following condition on $f$ from \cite[Definition 1.4]{FK}: \emph{All non-empty geometric fibres of $f$ are smooth of dimension $d$} or more generally condition $(\ast)_d$ from \cite[Chapitre XVIII, Théorème 2.9]{SGA4}: \emph{There exists an open $U\subseteq X$ such that $f|_U$ is flat with fibres of dimension $\leq d$ and the fibres over $X\setminus U$ have dimension $<d$.}

If $f$ satisfies $(\ast)_d$, then there is a trace map 
\[
\int_{X/S}\colon R^{2d}f_!\Lambda(d)\to\Lambda,
\]
where the $\Lambda$ are constant $n$-torsion sheaves on $X$ and $S$ respectively \cite[Theorem II.1.6]{FK} or \cite[Chapitre VXIII, Théorème 2.9]{SGA4}. Compatibility of the trace map with base change implies that the stalk at a point $s\in S$ of $\int_{X/S}$ equals the trace map $\int_{X_{\bar s}}\colon\mrH_c^{2d}(X_{\bar s},\Lambda(d))\to\Lambda$. 

There is a morphism of functors $t_f\colon f^\ast(d)[2d]\to Rf^!$ so that the diagram 
\[
\begin{tikzcd}
Rf_! f^\ast(d)[2d]\ar[d,"\int_{X/S}"]\ar[r,"t_f"]&Rf_! Rf^!\ar[ld,"Rf_!\vdash Rf^!"]\\
\id&~
\end{tikzcd}
\]
commutes, \cite[Chapitre XVIII, Lemme 3.2.3]{SGA4}. By \cite[Chapitre XVIII, Théorème 3.2.5]{SGA4} this $t_f$ is an isomorphism if $f$ is smooth, in particular, if we take $f\colon X\to S=\spec(\bar k)$ smooth, then the unit $Rf_!Rf^!\Lambda\to\Lambda$ is the trace morphism, which is (in degree 0) $\int_X\colon\mrH_c^{2d_X}(X,\Lambda)\isomto\Lambda$.

\begin{prop}[Poincaré Duality]
Let $X$ be smooth over $k=\bar k$ and let $\mathrm{char}(k)\nmid n$. For $\mathcal F$ a locally constant constructible sheaf of $\Z/n\Z$-modules, there is a natural isomorphism $\Het^{2d_X-i}(X,\mathcal F^\vee)\cong\mrH_c^i(X,\mathcal F)^\vee=\Hom_{\Z/n\Z}(\mrH_c^i(X,\mathcal F),\Z/n\Z)$.
\end{prop}
\begin{proof}
This follows either from the Verdier Duality \cite[Chapitre XVIII, (3.2.6.2)]{SGA4} or can be shown directly: \cite[Theorem II.1.13]{FK}, \cite[Chapter VI Theorem 11.1]{Milne}.
\end{proof}

\begin{definition}\label{Gysin map}
Let $f\colon X\to Y$ be a proper map between smooth varieties and set $c=d_X-d_Y$. We define the Gysin map $f_\ast\colon\Het^i(X,\Lambda)\to\Het^{i-2c}(Y,\Lambda)$ to be the composition
\[
\Het^i(X,\Lambda)\cong\mrH_c^{2d_X-i}(X,\Lambda)^\vee\overset{(f^\ast)^\vee}{\to}\mrH_c^{2d_X-i}(Y,\Lambda)^\vee\cong\mrH^{i-2c}(Y,\Lambda).
\]
\end{definition}

\subsubsection{Cup Product}
Let $\mathcal F$ be a sheaf on a scheme $S$ and $U\to S$ étale. If $\alpha\in\Gamma(U,\mathcal F)$ is a section of a sheaf and $\bar s\to U$ a geometric point, then we will write $\alpha_{\bar s}\in\mathcal F_{\bar s}$ for its image in the stalk. In particular, if $f\colon X\to S$ is compactifiable and $\alpha\in\Gamma(U,R^if_!\Lambda)$, then the base-change theorem gives $\alpha_{\bar s}\in \mrH^i_c(X_{\bar s},\Lambda)$.

Given an étale neighborhood $U\to S$ of $s\in S$, then we would like to have a cup product $\Gamma(U,R^if_!\Lambda)\times\Het^j(X_U,\Lambda)\overset{-\cup-}{\to}\Gamma(U,R^{i+j}f_!\Lambda)$ fitting in a commutative diagram
\begin{equation}\label{cup product}
\begin{tikzcd}
\Gamma(U,R^if_!\Lambda)\times\Het^j(X_U,\Lambda)\ar[d]\ar[r,"-\cup-"]&\Gamma(U,R^{i+j}f_!\Lambda)\ar[d]\\
\mrH_c^i(X_{\bar s},\Lambda)\times\Het^j(X_{\bar s},\Lambda)\ar[r,"-\cup-"]&\mrH_c^{i+j}(X_{\bar s},\Lambda)
\end{tikzcd}
\end{equation}
where the vertical maps are the natural stalk and restriction maps and the bottom map is the ordinary cup product, defined as follows (\cite[page 303]{FK}): let $\varphi\in\Hom_{D(X,\Lambda)}(\Lambda,\Lambda[j])=\Het^j(X,\Lambda)$, which induces $R(p_X)_!(\varphi)\colon R(p_X)_!(\Lambda)\to R(p_X)_!(\Lambda[j])$, where we write $p_X\colon X\to\spec(k)$ for the structure map, and taking $i$-th cohomology gives $\mrH_c^i(X,\Lambda)\to\mrH_c^{i+j}(X,\Lambda)$. 

Write $f^U:=f|_{X_U}\colon X_U\to U$ for the restriction to the étale neighborhood $U\to S$. Then any $\varphi\in\Het^j(X_U,\Lambda)=\Hom_{D(X_U,\Lambda)}(\Lambda,\Lambda[j])$ induces a map $Rf^U_!\varphi\colon Rf^U_!\Lambda\to Rf^U_!\Lambda[j]$. Taking $i$-th cohomology and sections gives us a morphism
\[
\Gamma(U,R^if_!\Lambda)\times\Het^j(X_U,\Lambda)\overset{\cup}{\to}\Gamma(U,R^{i+j}f_!\Lambda).
\]
To show that the desired diagram (\ref{cup product}) commutes, write $i\colon X_{\bar s}\to X_U$ for the geometric fibre, then the composition
\[
\Hom_{D(X_U,\Lambda)}(\Lambda,\Lambda[j])\to\Hom_{D(X_U,\Lambda)}(\Lambda,Ri_\ast i^\ast\Lambda[j])\cong\Hom_{D(X_{\bar s},\Lambda)}(\Lambda,\Lambda[j])
\]
is given by $i^\ast$ and induces the restriction map $\Het^j(X_U,\Lambda)\to\Het^j(X_{\bar s},\Lambda)$ on cohomology. This shows that $Rf^U_!\varphi(\tilde\alpha)_{\bar s}=\tilde\alpha_{\bar s}\cup i^\ast\varphi$, for any $\tilde\alpha\in\Gamma(U,R^if_!\Lambda)$ and $\varphi\in\Het^j(X_U,\Lambda)$, as wished.

\subsubsection{Norm Map}
We recall here the construction of the Norm map from \cite[0BD2]{Stacks}. Let $f\colon X\to Y$ be a finite flat map, then $f_\ast\Oc_X$ is a finite locally free $\Oc_Y$-module \cite[02KB]{Stacks}, so there exists a Zariski open cover $\mathcal U$ of $Y$ so that $(f_\ast\Oc_X)|_U\cong\Oc_U^{\oplus d}$ for every $U\in\mathcal U$. Then we can define the norm map $(f_\ast\Oc_X^\ast)|_U\to\Oc_U^\ast$ by the determinant of the corresponding multiplication matrix. This also respects localization, so it glues to a map $f_\ast\Oc_X^\ast\to\Oc_Y^\ast$. Note that this also extends over étale sheaves, thus we obtain the following.
\begin{definition}
Denote the norm map on étale sheaves $N\colon f_\ast\G_{m,X}\to\G_{m,Y}$.
\end{definition}

Let $f\colon B'\to B$ be a finite map between smooth curves. Note that such a map is automatically flat. For $\mathcal X\to B$ a smooth family, write $f\colon\mathcal X':=\mathcal X\times_BB'\to\mathcal X$ for the base change as well, which is thus also finite flat. For each Zariski open $U\subseteq B$ we have $U':=f^{-1}(U)\to U$ is still finite flat. By \cite[page 136]{FK} we have a commutative diagram involving the norm and trace map
\[
\begin{tikzcd}
0\ar[r]&f_\ast\mu_{\ell^r}\ar[r]\ar[d,"\int_{\mathcal X'_{U'}/\mathcal X_U}"]&f_\ast\Oc_{\mathcal X'_{U'}}^\ast\ar[r,"\ell^r"]\ar[d,"N"]&f_\ast\Oc_{\mathcal X'_{U'}}^\ast\ar[r]\ar[d,"N"]&0\\
0\ar[r]&\mu_{\ell^r}\ar[r]&\Oc_{\mathcal X_U}^\ast\ar[r,"\ell^r"]&\Oc_{\mathcal X_U}^\ast\ar[r]&0
\end{tikzcd}.
\]
The corresponding long exact sequence then gives a commutative diagram
\[
\begin{tikzcd}
\Het^0(\mathcal X_U,f_\ast\Oc_{\mathcal X'_{U'}}^\ast)\ar[r]\ar[d,"N"]&\Het^1(\mathcal X_U,f_\ast\mu_{\ell^r})\ar[d,"\int_{\mathcal X'_{U'}/\mathcal X_U}"]\\
\Het^0(\mathcal X_U,\Oc_{\mathcal X_U}^\ast)\ar[r]&\Het^1(\mathcal X_U,\mu_{\ell^r})
\end{tikzcd},
\]
taking the direct limit over all Zariski opens $U\subseteq B$, we obtain a commuting diagram
\[
\begin{tikzcd}
\Het^0(X_{\eta'},\Oc_{X_{\eta'}}^\ast)\ar[r]\ar[d,"N"]&\Het^1(X_{\eta'},\mu_{\ell^r})\ar[d,"\varinjlim_{U\subseteq B} f_\ast"]\\
\Het^0(X_\eta,\Oc^\ast_{X_\eta})\ar[r]&\Het^1(X_\eta,\mu_{\ell^r})
\end{tikzcd}.
\]
We obtain the following compatibility.
\begin{lemma}\label{Compatibility Norm and Gysin}
The following diagram commutes
\[
\begin{tikzcd}
\Het^0(\eta',\Oc^\ast_{\eta'})\ar[r, two heads]\ar[d,"N"]&\Het^1(\eta',\mu_{\ell^r})\ar[r]\ar[d,"\varinjlim f_\ast"]&\Het^1(X_{\eta'},\mu_{\ell^r})\ar[d,"\varinjlim f_\ast"]\\
\Het^0(\eta,\Oc^\ast_{\eta})\ar[r, two heads]&\Het^1(\eta,\mu_{\ell^r})\ar[r]&\Het^1(X_{\eta},\mu_{\ell^r})
\end{tikzcd}.
\]
\end{lemma}
\begin{proof}
The left square commutes by the beginning. As the first horizontal maps are surjective by Hilbert-90, to show that the rigth square commutes, it suffices to show that the `outer' square commutes. Note that this square decomposes as 
\[
\begin{tikzcd}
\Het^0(\eta',\Oc^\ast_{\eta'})\ar[r]\ar[d,"N"]&\Het^0(X_{\eta'},\mathcal O_{X_{\eta'}}^\ast)\ar[r]\ar[d,"N"]&\Het^1(X_{\eta'},\mu_{\ell^r})\ar[d,"\lim f_\ast"]\\
\Het^0(\eta,\Oc^\ast_{\eta})\ar[r]&\Het^0(X_\eta,\mathcal O_{X_\eta}^\ast)\ar[r]&\Het^1(X_{\eta},\mu_{\ell^r})
\end{tikzcd}.
\]
Now the right square commutes by the beginning and the left square commutes by a direct computation.
\end{proof}

\subsubsection{Residue Map}
Suppose $X$ is smooth and $i\colon Z\to X$ is a smooth closed subvariety of codimension $c$ with open complement $j\colon U\hookrightarrow X$. Then there exists a long exact sequence
\[
\cdots\to\Het^{i-2c}(Z,\Lambda)\overset{i_\ast}{\to} \Het^i(X,\Lambda)\overset{j^\ast}{\to} \Het^i(U,\Lambda)\to \Het^{i+1-2c}(Z,\Lambda)\overset{i_\ast}{\to}\cdots,
\]
See for example \cite[Remark VI.5.4]{Milne}. From now on, we shall refer to the connecting morphism as the \emph{residue map}, $Res\colon\Het^i(U,\Lambda)\to \Het^{i+1-2c}(Z,\Lambda)$. 

\begin{remark}\label{residue compatible restriction}
It is known that the residue map is compatible with the natural restriction to opens \cite[Lemma 2.3]{Sch2}.
\end{remark}

The following is \cite[Lemma 2.4]{SchNotes}, but we state it here again for reference purposes, together with a direct corollary.
\begin{prop}\label{cup-global-class}
Let $(X,Z)$ be a smooth pair and let $\alpha\in\Het^i(X\setminus Z,\Lambda)$ and $\beta\in\Het^i(X,\Lambda)$. Then $Res(\alpha\cup\beta|_{X\setminus Z})=Res(\alpha)\cup\beta|_Z$.
\end{prop}

\begin{cor}\label{product-residue}
Let $(X,Z)$ be a smooth pair with $Z$ closed of codimension $c$ and let $Y$ be any smooth variety. Let $\alpha\in\Het^i(X\setminus Z,\Lambda)$ and $\beta\in\Het^j(Y,\Lambda)$, then $Res(\alpha\times\beta)=Res(\alpha)\times\beta\in\Het^{i+j+1-2c}(Z\times Y,\Lambda)$.
\end{cor}

\subsection{Refined Unramified Cohomology}\label{ref unram cohom}
Here we set some notation and gather results from \cite{Sch1} which we will use throughout. We let $X$ be a smooth variety over a field $k$ so that we are in the situation of \cite[Lemma 6.5]{Sch1}. We let $\Lambda$ be $\mu_{\ell^r}$ or $\Z_\ell$ with $\ell$ prime to $\mathrm{char}(k)$.

\begin{definition}[\cite{Sch1}]
Let $j$ be an integer, we set
\begin{itemize}
\item $F_jX:=\{x\in X\mid \codim_X(\overline{\{x\}})\leq j\}$;
\item $\mrH^i(F_jX,\Lambda):=\varinjlim_{F_jX\subseteq U\subseteq X}\Het^i(U,\Lambda)$, where the direct limit is over all non-empty Zariski open $U$ so that $X\setminus U$ is of codimension $>j$.
\end{itemize}
\end{definition}
Note that for $j'\geq j$ we have natural maps $\mrH^i(F_{j'}X,\Lambda)\to\mrH^i(F_jX,\Lambda)$. We will usually abuse notation and call this map $F_j$, assuming that the domain is understood from the context. Now the refined unramified cohomology is defined as 
\[
\Hrnr{j}^i(X,\Lambda):=\im(\mrH^i(F_{j+1}X,\Lambda)\to\mrH^i(F_jX,\Lambda)).
\]

\begin{lemma}[{\cite[Lemma 5.8]{Sch1}}]\label{les Sch}
There is a long exact sequence 
\[
\cdots\to\bigoplus_{x\in X^{(j)}}\Het^{i-2j}(\kappa(x))\to\mrH^i(F_jX)\to\mrH^i(F_{j-1}X)\to\bigoplus_{x\in X^{(j)}}\Het^{i-2j+1}(\kappa(x))\to\Het^{i+1}(F_jX)\to\cdots,
\]
where all cohomology groups have coefficients $\Lambda$.
\end{lemma}

\begin{cor}[{\cite[Corollary 5.10]{Sch1}}]
The natural map $\Het^i(X,\Lambda)\to\mrH^i(F_jX,\Lambda)$ is an isomorphism if $j\geq\lceil\frac{i}{2}\rceil$.
\end{cor}

We will use the following interpretation of the failure of the integral Hodge/Tate type conjecture in terms of refined unramified cohomology.

\begin{theorem}[{\cite[Theorem 7.7]{Sch1}}]
Write $Z^i(X):=\coker(\cl^i\colon \CH^i(X)_{\Z_\ell}\to\Het^{2i}(X,\Z_\ell(i))$, then 
\[
Z^i(X)[\ell^r]\cong\frac{\Hrnr{i-2}^{2i-1}(X,\mu_{\ell^r})}{\Hrnr{i-2}^{2i-1}(X,\Z_\ell)}.
\]
\end{theorem}

\begin{remark}\label{Tate classes}
Note that the non-vanishing of $Z^i(X)[\ell^r]$ indeed gives the failure of the integral Tate \Cref{ITC}. After all, suppose $0\neq\alpha\in Z^i(X)[\ell^r]$, then $\alpha$ is non-algebraic but $\ell^r\alpha=\cl^i(\Gamma)$ is. This implies that $\ell^r\alpha\in\Het^{2i}(X,\Z_\ell(i))^{G_{k_0}}$ for some $k_0\subseteq k$ over which $\Gamma$ is defined. We would now like to show that there exists a finite field extension $k_0\subseteq k_0'$ so that $\alpha\in\Het^{2i}(X,\Z_\ell(i))^{G_{k'}}$. It holds that $g\alpha-\alpha\in\Het^{2i}(X,\Z_\ell(i))[\ell^r]$ for every $g\in G_{k_0}$. So $G_{k_0}$ acts on the finite set $\alpha+\Het^{2i}(X,\Z_\ell(i))[\ell^r]$. As a profinite set is Hausdorff, $G_{k_0}$ now acts on a discrete space, so it has open stabalizers. This means there exists a finite field extension $k_0\subseteq k'$ so that $G_{k'}$ stablizes $\alpha$, that is $\alpha\in\varinjlim_{k_0\subseteq k'}\Het^{2i}(X,\Z_\ell(i))^{G_{k'}}$ and is non-algebraic.
\end{remark}

We will also use the following properties. They essentially follow from the known properties on ordinary cohomology (ie. for $j$ large), \cite[09YQ and 01YZ]{Stacks}.
\begin{lemma}\label{refined-geometric-point}
Let $X$ be defined over a field $\eta$ (ie. the generic fibre of a family) and $\Lambda=\mu_{\ell^r}$ be finite. Then there is a natural isomorphism $\varinjlim_{\eta\subseteq\eta'}\mrH^i(F_jX_{\eta'},\Lambda)\to\mrH^i(F_jX_{\bar\eta},\Lambda)$ compatible with the $F_j$ filtration, where $\bar\eta$ is the algebraic (resp. seperable) closure of $\eta$ and the direct limit runs over all finite (resp. finite separable) extensions $\eta'$ of $\eta$.
\end{lemma}
\begin{proof}
All cohomology groups below have coefficients in $\Lambda$.

Let $\eta'$ be a finite extension of $\eta$, we start by defining the map $\Het^i(F_jX_{\eta'})\to\mrH^i(F_jX_{\bar\eta})$. Let $\alpha\in\mrH^i(F_jX_{\eta'})$, then $\alpha$ is represented by some $\alpha_{U}\in\mrH^i(U)$ with $Z:=X_{\eta'}\setminus U$ of codimension $>j$. As the dimensions do not change under under algebraic field extensions \cite[Proposition 3.2.7]{Liu}, we have that $Z_{\bar\eta}=X_{\bar\eta}\setminus U_{\bar\eta}$ is still of codimension $>j$. We let the image of $\alpha$ in $\Het^i(F_jX_{\eta'})\to\mrH^i(F_jX_{\bar\eta})$ to be the image of $\alpha_U$ under the composition $\Het^i(U)\to\Het^i(U_{\bar\eta})\to\mrH^i(F_jX_{\bar\eta})$. This is well-defined because restricting to opens commutes with base extension. It is also natural with respect to compositions of base-extensions by construction. Thus this defines a map on the direct limits $\varinjlim_{(\eta'\to\eta)}\mrH^i(F_jX_{\eta'})\to\mrH^i(F_jX_{\bar\eta})$.

We start by showing it is surjective. Let $\alpha\in\mrH^i(F_jX_{\bar\eta})$, then $\alpha=\alpha_U\in\Het^i(U)$ for some $F_jX_{\bar\eta}\subseteq U\subseteq X_{\bar\eta}$, say with complement $Z$ of codimension $>j$. Then by \cite[Lemma 3.2.6]{Liu} there exists a finite extension $\eta'\to\eta$ and $Z'\subseteq X_{\eta'}$ such that $Z'_{\bar\eta}=Z$. So if we put $U':=X_{\eta'}\setminus Z'$, then $U'_{\bar\eta}=U$. We know that $\varinjlim_{\eta''\to\eta'}\Het^i(U'_{\eta''})\isomto\Het^i(U'_{\bar\eta})=\Het^i(U)$, so $\alpha_U=\alpha_{U'_{\eta''}}$ for some $\eta''$ (\cite[09YQ and 01YZ]{Stacks}). Using again \cite[Proposition 3.2.7]{Liu} we see that $\dim(Z'_{\eta''})=\dim(Z')=\dim(Z)$, so $F_jX_{\eta''}\subseteq U'_{\eta''}\subseteq X_{\eta''}$. Then by construction, the image of $\alpha_{U'_{\eta''}}$ in $\mrH^i(F_jX_{\eta''})$ maps to $\alpha\in\mrH^i(F_jX_{\bar\eta})$.

For injectivity, let $\alpha\in\mrH^i(F_jX_{\eta'})$ and suppose it maps to $0$. We will show that $\alpha=0\in\varinjlim_{\eta''\to\eta'}\mrH^i(F_jX_{\eta''})$ and for the sake of notation, we write $\eta=\eta'$. By definition, this means that if $\alpha$ is represented by $\alpha_U$, then there exists an open $F_jX_{\bar\eta}\subseteq V\subseteq U_{\bar\eta}\subseteq X_{\bar\eta}$ such that $\alpha_{U_{\bar\eta}}|_V=0$. For a similar reason as above, using again \cite[Lemma 3.2.6]{Liu} there exists a finite extension $\eta'\to\eta$ and $V'\subseteq X_{\eta'}$ such that $V'_{\bar\eta}=V$ and $V'\subseteq U_{\eta'}=:U'$. So $0=\alpha_{U_{\bar\eta}}=(\alpha_{U'}|_{V'})_{\bar\eta}$ and thus there exists a finite extension $\eta''\to\eta'$ such that $\alpha_{U_{\eta''}}|_{V'_{\eta''}}=(\alpha_{U'}|_{V'})_{\eta''}=0$, but for the same reason as above, $F_jX_{\eta''}\subseteq V'_{\eta''}\subseteq U_{\eta''}\subseteq X_{\eta''}$, implying that $F_j(\alpha_{U_{\eta''}})=0\in \mrH^i(F_jX_{\eta''})$ as desired.
\end{proof}

\begin{lemma}\label{refined-generic-point}
Let $\mathcal X\to Y$ be a family over a smooth irreducible variety $Y$ with generic point $\eta$. There is a natural isomorphism $\varinjlim_{U\subseteq Y}\mrH^i(F_j\mathcal X_U,\Lambda)\to\mrH^i(F_j X_\eta,\Lambda)$, where the direct limit runs over all Zariski opens $U\subseteq Y$. 
\end{lemma}
\begin{proof}
Let $\alpha\in\varinjlim_{U\subseteq Y}\mrH^i(F_j\mathcal X_U,\Lambda)$ be represented by some $\alpha_U\in\mrH^i(F_j\mathcal X_U,\Lambda)$. Say this $\alpha_U$ is represented by some $\alpha_{U,V}\in \Het^i(V,\Lambda)$ for some $F_j\mathcal X_U\subseteq V\subseteq\mathcal X_U$ and let $Z:=\mathcal X_U\setminus V$. By generic flatness, there exists a $U'\subseteq U$ so that the induced map $Z\cap\mathcal X_{U'}\to U'$ is flat. This implies that $\dim(Z_\eta)+\dim(Y)=\dim(Z\cap\mathcal X_{U'})\leq\dim(Z)$. So $\dim(Y_\eta)-\dim(Z_\eta)=\dim(\mathcal X)-\dim(Y)-\dim(Z_\eta)\geq\dim(\mathcal X)-\dim(Z)>j$. This means that $V_\eta:=Z\cap X_\eta$ has the property that $F_jX_\eta\subseteq V_\eta$, so we may define the image of $\alpha$ to be $[\alpha_{U,V}|_{V_\eta}]\in\mrH^i(F_jX_\eta,\Lambda)$. As this is independent of chosen $U$ and $V$, this defines a morphism $\varinjlim_{U\subseteq Y}\mrH^i(F_j\mathcal X_U)\to\mrH^i(F_j X_\eta,\Lambda)$.

To show it is an isomorphism, we construct its inverse. Let $\alpha\in\mrH^i(F_jX_\eta,\Lambda)$ and say it is represented by some $\alpha_V\in\Het^i(V,\Lambda)$ with $F_jX_\eta\subseteq V\subseteq X_\eta$ with complement $Z:=X_\eta\setminus V$. Let $\mathcal Z$ be the closure of $Z$ in $\mathcal X$. So we have an induced map $\mathcal Z\to Y$ and by generic flatness again, there exists an open $U\subseteq Y$ so that $\mathcal Z_U\to U$ is flat. Write $\mathcal V_U:=\mathcal X_U\setminus\mathcal Z_U$, then note that $(\mathcal V_U)_\eta=V$. As $\alpha_V\in\Het^i(V,\Lambda)=\varinjlim_{U'\subseteq U}\Het^i(\mathcal V_{U'},\Lambda)$ (\cite[09YQ and 01YZ]{Stacks}), there exists an open $U'\subseteq U$ so that $\alpha_V$ comes from some $\alpha_{V,U'}\in\Het^i(\mathcal V_{U'},\Lambda)$. Note that $\mathcal Z_{U'}\to U'$ is still flat, so we have $\dim(\mathcal X_{U'})-\dim(\mathcal Z_{U'})=\dim(\mathcal X_{U'})-\dim(U')-\dim(Z)=\dim(X_\eta)-\dim(Z)>j$. So if we define the image of $\alpha$ to be the image of $\alpha_{V,U'}$ under the map $\Het^i(\mathcal V_{U'},\Lambda)\to\Het^i(F_j\mathcal X_{U'},\Lambda)\to\varinjlim_{U\subseteq Y}\mrH^i(F_j\mathcal X_U)$, then one can check that this defines the inverse of the map above.
\end{proof}

We will also need the following, which is an easier version of \cite[Lemma 7.8]{Sch1}.
\begin{lemma}\label{liftable lemma}
Let $r\in\Z_{>0}$ and  $A=\mu_{\ell^s}$ with $s\geq r$. Then the natural map
\[
\frac{\Het^{2i}(X,\mu_{\ell^r})}{\Het^{2i}(X,A)}=\frac{\Het^{2i}(F_iX,\mu_{\ell^r})}{\Het^{2i}(F_iX,A)}\overset{F_{i-1}}{\to}\frac{\Hrnr{i-1}^{2i}(X,\mu_{\ell^r})}{\Hrnr{i-1}^{2i}(X,A)}
\]
is an isomorphism.
\end{lemma}
\begin{proof}
This follows directly from a chase in the commuting diagram with exact rows
\[
\begin{tikzcd}
\bigoplus_{x\in X^{(i)}}x\cdot A\ar[r,"\cl^i"]\ar[d, two heads,"\mod \ell^r"]&\mrH^{2i}(X,A)\ar[r,"F_{i-1}"]\ar[d]&\Hrnr{i-1}^{2i}(X,A)\ar[d]\ar[r]&0\\
\bigoplus_{x\in X^{(i)}}x\cdot\mu_{\ell^r}\ar[r,"\cl^i"]&\mrH^{2i}(X,\mu_{\ell^r})\ar[r,"F_{i-1}"]&\Hrnr{i-1}^{2i}(X,\mu_{\ell^r})\ar[r]&0
\end{tikzcd}.
\]
Let $\alpha\in\mrH^{2i}(X,\mu_{\ell^r})$ and suppose $F_{i-1}(\alpha)=F_{i-1}(\beta)\mod\ell^r=F_{i-1}(\beta\mod\ell^r)$ for some $\beta\in\mrH^{2i}(X,A)$, then $\alpha-(\beta\mod\ell^r)$ is algebraic and thus equal to $\cl^i(\xi\mod\ell^r)$ for some $\xi\in\bigoplus_{x\in X^{(i)}}x\cdot A$ by the surjectivity of the left vertical map. Then we compute that indeed $(\cl^i(\xi)+\beta)\mod\ell^r=\alpha$.
\end{proof}

\subsection{Specialization Map}\label{def spec map}
In this section we recall the construction of the specialization map from \cite[Section 4]{Sch2} and recollect some of its properties.

\subsubsection{The Construction}
Let $\mathcal X\to B$ be a smooth family over a smooth curve $B$. Let $\eta\in B$ be the generic point and $b\in B$ a closed point. Note that we have a residue map $\Het^i(X_\eta,\Lambda)=\varinjlim_{U\subseteq B}\Het^i(\mathcal X_U,\Lambda)\overset{Res}{\to}\bigoplus_{b\in B}\Het^{i-1}(X_b,\Lambda)$, where the direct limit is over the Zariski open subsets of $B$. Define $Res_b\colon\Het^i(X_\eta,\Lambda)\to\Het^{i-1}(X_b,\Lambda)$ to be the above residue map composed with the projection to the $b$-th component.

For the construction of the specialization map, note that $\Oc_{B,b}$ is a disrete valuation ring and let $\pi\in Frac(\Oc_{B,b})=\eta$ be a uniformizer. We can view this as an element $\pi\in\eta^\ast/\eta^{\ast\ell^r}=\Het^1(\eta,\mu_{\ell^r})$ and shall write $(\pi):=f^\ast\pi\in\Het^1(X_\eta,\mu_{\ell^r})$.
\begin{definition}
The specialization map $sp_b\colon\Het^i(X_\eta,\mu_{\ell^r})\to\Het^i(X_b,\mu_{\ell^r})$ is defined as 
\[
sp_b(\alpha):=-Res_b((\pi)\cup\alpha).
\]
\end{definition}
This is independent of chosen uniformizer, for if $\pi'$ is another uniformizer, then $\pi-\pi'$ is a unit, hence has valuation zero. As the residue map $\Het^1(\eta,\mu_{\ell^r})\to\Het^0(b,\mu_{\ell^r})\cong\Z/\ell^r\Z$ is given by taking the valuation \cite[Proposition 1.3]{CT-O}, we see that this implies that $\pi-\pi'$ has residue $0\in\Het^0(b,\mu_{\ell^r})$. So we are done by the following lemma.

\begin{lemma}\label{Specialization Along Liftable}
Let $\pi\in\Het^1(\eta,\mu_{\ell^r})$ such that $Res_b(\pi)=0$ in $\Het^{0}(b,\mu_{\ell^r})$, then $Res_b((\pi)\cup\alpha)=0$ in $\Het^i(X_b,\mu_{\ell^r})$ for every $\alpha\in\Het^i(X_\eta,\mu_{\ell^r})$. 

Moreover, if $\pi\in\Het^j(\eta,\mu_{\ell^r})$ for $j>1$, then $Res_b((\pi)\cup\alpha)=0$.
\end{lemma}
\begin{proof}
As $\Het^{j-1}(b,\mu_{\ell^r})=0$ for $j>1$, both assumptions imply that $\pi$ lifts to some $\mu\in\Het^j(\Oc_{B,b},\mu_{\ell^r})$, thus by \Cref{cup-global-class} we have $Res_b((\mu)|_{X_\eta}\cup\alpha)=(-1)^j(\mu)|_{X_b}\cup Res_b(\alpha)$. But $(\mu)|_{X_b}=(\mu_b)=0$ as $\Het^j(b)=0$ for $j>0$.
\end{proof}

We also note the following property, which is a consequence of \Cref{cup-global-class}.
\begin{prop}[{\cite[Lemma 4.4]{Sch2}}]\label{global sp}
If $\alpha\in \Het^i(X_\eta,\mu_{\ell^r})$ is extendable to some $\alpha_U\in \Het^i(\mathcal X_U,\mu_{\ell^r})$ with $b\in U$, then $sp_b(\alpha)=\alpha_U|_{X_b}$.
\end{prop}

At some point, we will need the following property of the specialization map.
\begin{prop}
Let $\pi$ be a uniformizer at $b$, then $\im(\Het^i(X_\eta,\mu_{\ell^r})\overset{(\pi)\cup -}{\to}\Het^{i+1}(X_\eta,\mu_{\ell^r}))\subseteq\ker(sp_b)$
\end{prop}
\begin{proof}
\Cref{Specialization Along Liftable} implies directly that $Res_b((\pi\cup\pi)\cup\alpha)=0$.
\end{proof}

\begin{remark}\label{deRham1}
Let us note the similarity with the residue map on logarithmic deRham cohomology as in \cite[\textsection 4.1]{PS}. If we let $t$ be a local coordinate around $b$, then we can set $sp(\omega):=-res(\frac{\mathrm{d}t}{t}\wedge\omega)$, where now $\frac{\mathrm{d}t}{t}$ serves as the uniformizer. So we clearly see that $sp(\frac{\mathrm{d}t}{t}\wedge\omega)=0$.
\end{remark}

\begin{prop}\label{product-sp}
Let $\mathcal X\to B$ be a smooth family over a curve and $Y$ any smooth variety. Let $\alpha\in\Het^i(X_\eta,\mu_{\ell^r})$ and $\beta\in\Het^j(Y,\mu_{\ell^r})$, then $sp_b(\alpha\times\beta)=sp_b(\alpha)\times\beta\in\Het^{i+j}(X_b\times Y,\mu_{\ell^r})$.
\end{prop}
\begin{proof}
This follows directly from \Cref{product-residue}.
\end{proof}

\begin{remark}\label{coeff sp}
We note that the specialization map is also functorial in its coefficients with respect to maps $\mu_{\ell^r}\twoheadrightarrow\mu_{\ell^s}$ with $r\geq s$.
\end{remark}

The following lemma implies that the specialization map can be defined on the $F_j$ filtrations, which is also \cite[Lemma 4.6]{Sch2}.
\begin{lemma}\label{Residue Refined}
Let $X$ be a smooth variety and $D\subseteq X$ a smooth divisor. There exists a refined residue morphism $\mrH^i(F_n(X\setminus D),\Lambda)\overset{Res}{\to}\mrH^{i-1}(F_n D,\Lambda)$ compatible with the filtrations $F_\ast$.
\end{lemma}
\begin{proof}
Let $F_n(X\setminus D)\subseteq U\subseteq X\setminus D$ be a Zariski open, so $Z:=(X\setminus D)\setminus U$ is of codimension $>n$ in $X\setminus D$. Let $\bar Z$ be the closure of $Z$ in $X$. Then we compute that $U=X\setminus (D\cup\bar Z)=U'\setminus D'$, where $U':=X\setminus \bar Z$ and $D'=D\cap U'=D\setminus D\cap \bar Z$. Note that $D'\neq\emptyset$, else $D\subseteq\bar Z$, but this implies that either $\bar Z=X$, so $Z=X\setminus D$ which is not possible, or there exists an irreducible component $Z_0\subseteq Z$ such that $\bar Z_0=D$, but $Z_0\cap D\subseteq Z\cap D=\emptyset$. So using \cite[0A21]{Stacks}, we conclude that $D'$ is a smooth divisor of the smooth $U'$ with open complement $U$. So we have the residue map $\Het^i(U,\Lambda)\to\Het^{i-1}(D',\Lambda)$.

We will show that $D\setminus D'$ has codimension $>n$ in $D$ so that we have $\Het^{i-1}(D',\Lambda)\to \mrH^{i-1}(F_nD,\Lambda)$. As before, no irreducible component of $\bar Z$ is contained in $D$, so $\bar Z$ intersects $D$ transversally because $D$ is a divisor (\cite[begin \textsection 13.1.1]{EisHar} or \cite[above Lemma 9.9]{VoisII}) or $\bar Z\cap D=\emptyset$ in which case $D=D'$. This means that $\dim(\bar Z)-1=\dim(D\cap\bar Z)=\dim(D\setminus D')$ and thus $\dim(D)-\dim(D\setminus D')=\dim(D)-\dim(\bar Z)+1=\dim(X)-\dim(\bar Z)$. But $X\setminus D\subseteq X$ and $Z=\bar Z\cap (X\setminus D)\subseteq \bar Z$ are opens, so again by \cite[0A21]{Stacks} we have $\dim(X)-\dim(\bar Z)=\dim(X\setminus D)-\dim(Z)>n$ as wished.

As the residue map is compatible with restriction to opens, the above gives a well-defined map $\mrH^i(F_n(X\setminus D),\Lambda)\to\mrH^{i-1}(F_nD,\Lambda)$.

\end{proof}

Using the above Lemma and \Cref{refined-generic-point}, we can now consider the residue map $\mrH^i(F_n X_\eta,\Lambda)\cong\varinjlim_{U\subseteq B}\mrH^i(F_n\mathcal X_U,\Lambda)\to\bigoplus_{b\in B}\mrH^{i-1}(F_nX_b,\Lambda)$. Write again $Res_b\colon\mrH^i(F_nX_\eta,\Lambda)\to\mrH^{i-1}(F_nX_b,\Lambda)$ for the composition of this map with the projection to the $b$-th component. Using this, we can define the specialization map on $\Het^i(F_nX_\eta,\mu_{\ell^r})$ for every $n$.

\begin{cor}[{\cite[Lemma 4.6]{Sch2}}]\label{Specialization Refined}
For each $n$ there exists a map $sp_b\colon\Het^i(F_nX_\eta,\mu_{\ell^r})\to\Het^i(F_nX_b,\mu_{\ell^r})$ compatible with the filtration $F_\ast$.
\end{cor}
\begin{proof}
Let $\alpha\in\Het^i(F_nX_\eta,\mu_{\ell^r})$, then $\alpha=[\alpha_U]$ for some $\alpha_U\in\Het^i(U,\mu_{\ell^r})$ with $X_\eta\setminus U$ of codimension $>n$. We write $(\pi)\cup\alpha:=[(\pi)|_U\cup\alpha_U]\in\Het^{i+1}(F_nX_\eta)$ and set
\[
sp_b(\alpha):=-Res_b((\pi)\cup\alpha)\in \Het^i(F_nX_b,\Lambda),
\]
which is precisely \cite[Lemma 4.6]{Sch2}. Similar as before, one can check that it is still independent of chosen uniformizer.
\end{proof}

\subsubsection{Compatibility with Finite Extensions}
Let $\mathcal X\to B$ be a smooth family over a smooth curve $B$ and $f\colon B'\to B$ be a smooth finite morphism. Write also $f\colon\mathcal X'\to\mathcal X$ for the base change. Let $b'\in B'$ and $b=f(b')\in B$. Because we work over an algebraically closed field, $f$ induces an isomorphism $f_{b'}\colon X'_{b'}\isomto X_b$.

\begin{lemma}\label{Gysin Respects Uniformizer}
Suppose $f^{-1}(b)=\{b',b_1',\dots,b_n'\}$, then there exists a uniformizer $\pi'\in\Oc_{B',b'}$ that does not vanish at any of the other $b_i'$. Moreover, for such $\pi'$ we have that $f_\ast\pi'\in\Het^1(\eta,\mu_{\ell^r})$ represents a uniformizer of $b$.
\end{lemma}
\begin{proof}
Embed $B'\subseteq\P^N$ in some projective space. Let $H_0,H_1$ be hyperplanes of $\P^N$, intersecting $B$ transversally and so that $H_0\cap f^{-1}(b)=\emptyset$ and $H_1\cap f^{-1}(b)=\{b'\}$. Then inside $\P^N\setminus H_0\cong\A^N$, the hyperplane $H_1$ defines a function that vanishes at $b'$ but not at any of the other $b_i'$.

Suppose this $\pi'$ is defined on some $U'\subseteq B'$. Let $U:=B\setminus f(B'\setminus U')$ and shrink $U'$ to be $f^{-1}(U)$. Note that now $b\notin U$ as $b'\notin U'$ and by \cite[Lemma 2.3]{Sch2} we have a commutative diagram
\[
\begin{tikzcd}
\Het^1(U')\ar[r,"Res"]\ar[d,"f_\ast"]&\bigoplus_{b'\in B'\setminus U'}\Het^0(b')\ar[d,"f_\ast"]\\
\Het^1(U)\ar[r,"Res"]&\bigoplus_{b\in B\setminus U}\Het^0(b)
\end{tikzcd}.
\]
So we see that $Res_b(f_\ast\pi')=\sum_{b'\in f^{-1}(b)}f_\ast Res_{b'}(\pi')=f_\ast Res_{b'}(\pi')$ by construction of $\pi'$. Because $\pi'$ is a uniformizer at $b'$, we conclude that $f_\ast\pi'$ represents a uniformizer at $b$.
\end{proof}

\begin{prop}\label{compatibility-sp-fin-ext}
The diagram
\[
\begin{tikzcd}
\Het^i(X_{\eta'},\mu_\ell)\ar[r,"sp_{b'}"]&\Het^i(X'_{b'},\mu_\ell)\ar[d,"(f_{b'})_\ast"]\\
\Het^i(X_{\eta},\mu_\ell)\ar[u,"f^\ast"]\ar[r,"sp_b"]&\Het^i(X_b,\mu_\ell)
\end{tikzcd}
\]
commutes, that is, $(f_{b'})_\ast sp_{b'}(f^\ast\alpha)=sp_b(\alpha)$.
\end{prop}
\begin{proof}
Using \cite[Lemma 2.3]{Sch2} we have a commutative diagram
\[
\begin{tikzcd}
\Het^{i+1}(\mathcal X')\ar[r]\ar[d,"f_\ast"]&\Het^{i+1}(X_{\eta'})\ar[r,"Res"]\ar[d,"f_\ast"]&\bigoplus_{b'\in B'}\Het^i(X'_{b'})\ar[d,"(\sum_{b'\in f^{-1}(b)}(f_{b'})_\ast)_b"]\\
\Het^{i+1}(\mathcal X)\ar[r]&\Het^{i+1}(X_\eta)\ar[r,"Res"]&\bigoplus_{b\in B}\Het^i(X_{b})
\end{tikzcd},
\]
where the middle vertical map must be viewed as the direct limit of push forwards.

Let $\pi'$ be a uniformizer at $b'$  with the property from \Cref{Gysin Respects Uniformizer}, so that $f_\ast\pi'=:\pi$ is a uniformizer at $b$. Then we compute
\[
-Res_b(f_\ast((\pi')\cup f^\ast\alpha))=-Res_b((\pi)\cup\alpha)=sp_b(\alpha),
\]
using the projection formula and the second part of \Cref{Gysin Respects Uniformizer}. And on the other hand
\[
-Res_b(f_\ast((\pi')\cup f^\ast\alpha))=-\sum_{b'\in f^{-1}(b)}(f_{b'})_\ast Res_{b'}((\pi')\cup f^\ast\alpha)=(f_{b'})_\ast sp_{b'}(f^\ast\alpha),
\]
where in the last step we used the property of $\pi'$ and \Cref{Specialization Along Liftable} again.
\end{proof}

\begin{remark}
Note that \Cref{compatibility-sp-fin-ext} also implies \cite[Lemma 4.5]{Sch2} as if $\alpha,\tilde\alpha\in\Het^i(X_\eta,\mu_{\ell^r})$ so that $\alpha=\tilde\alpha\in\Het^i(X_{\bar\eta},\mu_{\ell^r})=\varinjlim_{\eta'\to\eta}\Het^i(X_{\eta'},\mu_{\ell^r})$, then there exists a finite extension $\eta'\to\eta$ so that $\alpha=\tilde\alpha\in\Het^i(X_{\eta'},\mu_{\ell^r})$. Let $b'\mapsto b$, then we compute $sp_b(\alpha)=(f_{b'})_\ast sp_{b'}(\alpha)=(f_{b'})_\ast sp_{b'}(\tilde\alpha)=sp_b(\tilde\alpha)$.
\end{remark}

\section{Comparing the Specialization and Cospecialization Map}\label{comparing sp and cosp}
The goal of this section is to give a relation between the cospecialization map as defined in \cite{FK} and the specialization map from \Cref{def spec map}. We start by recalling the cospecialization map.

\subsection{The Cospecialization map}\label{cospecialization map}
The following definition is a special case of the construction given on \cite[page 95]{FK}.
\begin{definition}
Let $\mathcal F$ be a sheaf on an irreducible variety $X$ with generic point $\eta$. For every $x\in X$, there is a cospecialization map
\[
cosp\colon\mathcal F_{\bar x}\to\mathcal F_{\bar\eta},
\]
defined by the inclusion of directed sets $\{(U,u)\overset{\e t}{\to}(X,x)\}\to\{U\overset{\e t}{\to}X\}$.
\end{definition}

If $f\colon X\to S$ is a compactifiable morphism and taking $\mathcal F=R^i f_!\Lambda$, then the base change isomorphism gives a homomorphism
\[
cosp\colon\mrH_c^i(X_{\bar s},\Lambda)\to\mrH_c^i(X_{\bar\eta},\Lambda).
\]
If $f$ is proper, we obtain
\[
cosp\colon\Het^i(X_{\bar s},\Lambda)\to\Het^i(X_{\bar\eta},\Lambda).
\]

The following property of the cospecialization map follows directly from the definition. We state it for reference purposes.
\begin{prop}\label{functorial-cosp}
Given a morphism of sheaves $\varphi\colon\mathcal F\to\mathcal G$ on $X$, then for every $x\in X$ the following diagram commutes
\[
\begin{tikzcd}
\mathcal F_{\bar x}\ar[r,"cosp"]\ar[d,"\varphi_{\bar x}"]&\mathcal F_{\bar\eta}\ar[d,"\varphi_{\bar\eta}"]\\
\mathcal G_{\bar x}\ar[r,"cosp"]&\mathcal G_{\bar\eta}
\end{tikzcd}.
\]
\end{prop}

\begin{cor}\label{constant trace}
Let $f\colon X\to S$ be a morphism satisfying $(\ast)_d$ (cf. \Cref{verdier duality}) and let $U\to S$ be an étale neighborhood where $U$ is irreducible with generic point $\eta$. Then for $\tilde\alpha\in\Gamma(U,R^{2d}f_!\Lambda)$ and $s\in U$, we have $\int_{X_{\bar\eta}}\tilde\alpha_{\bar\eta}=cosp\int_{X_{\bar s}}\tilde\alpha_{\bar s}$.
\end{cor}
\begin{proof}
This follows because the trace map is compatible with base change and by definition we see that $cosp(\tilde\alpha_{\bar s})=\tilde\alpha_{\bar\eta}$, so it follows from \Cref{functorial-cosp}.
\end{proof}

\begin{cor}\label{relation cosp}
Let $j\colon U\hookrightarrow X$ be an open and $f\colon X\to S$ a compactifiable morphism. Then the diagram
\[
\begin{tikzcd}
\mrH_c^i(U_{\bar s},\Lambda)\ar[r,"cosp"]\ar[d]&\mrH_c^i(U_{\bar\eta},\Lambda)\ar[d]\\
\mrH_c^i(X_{\bar s},\Lambda)\ar[r,"cosp"]&\mrH_c^i(X_{\bar\eta},\Lambda)
\end{tikzcd}
\]
commutes.
\end{cor}
\begin{proof}
The vertical morphisms are defined by taking the stalks of the cohomology of the composition $R^i(f|_U)_!\Lambda=R^if_!j_!j^\ast\Lambda\overset{j_!j^\ast\to\id}{\to}R^if_!\Lambda$. The result now follows from \Cref{functorial-cosp}.
\end{proof}

\begin{prop}\label{cosp-base-point-indep}
Let $f\colon\mathcal X\to B$ be a smooth family over a curve and $\pi\colon B'\to B$ a finite map. Let $b'\in B'$ be a closed point lying over $b\in B$. Then the following diagram commutes
\[
\begin{tikzcd}
\mrH_c^i(X_b,\Lambda)\ar[r,"cosp"]\ar[d,"\pi_{b'}^\ast"]&\mrH_c^i(X_{\bar\eta},\Lambda)\ar[d]\\
\mrH_c^i(X'_{b'},\Lambda)\ar[r,"cosp"]&\mrH_c^i(X_{\bar\eta'},\Lambda)
\end{tikzcd}.
\]
\end{prop}
\begin{proof}
Consider the diagram
\[
\begin{tikzcd}
\mrH_c^i(X_b,\Lambda)=(R^if_!\Lambda)_b\ar[r,"\cong"]\ar[d,"cosp"]&(\pi^\ast R^if_!\Lambda)_{b'}\ar[r,"\cong"]\ar[d,"cosp"]&(R^if'_!\Lambda)_{b'}\ar[d,"cosp"]=\mrH_c^i(X'_{b'},\Lambda)\\
\mrH_c^i(X_{\bar\eta},\Lambda)=(R^if_!\Lambda)_{\bar\eta}\ar[r,"\cong"]&(\pi^\ast R^if_!\Lambda)_{\bar\eta'}\ar[r,"\cong"]&(R^if'_!\Lambda)_{\bar\eta'}=\mrH_c^i(X_{\bar\eta'},\Lambda)
\end{tikzcd},
\]
where the right horizontal maps are the base-change isomorphisms. So the right square commutes by \Cref{functorial-cosp}. One can check directly that the left square commutes.

So we are left to check that the horizontal maps are the pullback maps on cohomology. This follows from formal compatibilities between the base-change isomorphisms. 
\end{proof}

Let us from now on assume that $\Lambda$ is finite of order prime to the characteristic of the base field $k$.
\begin{definition}\label{def dual cosp}
Let $f\colon\mathcal X\to B$ be a smooth family over a curve. We define 
\[
cosp^\vee\colon\Het^i(X_{\bar\eta},\Lambda)\to\Het^i(X_b,\Lambda),
\]
to be the composition of $(\mrH_c^{2d-i}(X_b,\Lambda)\overset{cosp}{\to}\mrH_c^{2d-i}(X_{\bar\eta},\Lambda))^\vee$ and the Poincaré Duality isomorphisms, where we use the identification
\[
\begin{tikzcd}
\mrH^{2d}_c(X_b,\Lambda)\ar[r,"cosp"]\ar[d,"\int_{X_b}"]&\mrH^{2d}_c(X_{\bar\eta},\Lambda)\ar[d,"\int_{X_{\bar\eta}}"]\\
\Lambda_b\ar[r,"cosp"]&\Lambda_{\bar\eta}
\end{tikzcd}.
\]
So $cosp^\vee(\alpha)$ is characterized by $cosp\int_{X_b}\beta\cup cosp^\vee(\alpha)=\int_{X_{\bar\eta}}cosp(\beta)\cup\alpha$ for every $\beta\in \mrH^{2d-i}_c(X,\Lambda)$.
\end{definition}

The following is also a direct consequence from \Cref{cosp-base-point-indep}.
\begin{cor}\label{dualcosp-base-point-indep}
The same setup as \Cref{cosp-base-point-indep}. The following diagram commutes
\[
\begin{tikzcd}
\Het^i(X_{\bar\eta'},\Lambda)\ar[r,"cosp^\vee"]\ar[d]&\Het^i(X_{b'},\Lambda)\ar[d,"(\pi_{b'})_\ast"]\\
\Het^i(X_{\bar\eta},\Lambda)\ar[r,"cosp^\vee"]&\Het^i(X_b,\Lambda)
\end{tikzcd}.
\]
\end{cor}

\subsection{General Comparison}
In this subsection we study under which conditions the specialization map from \Cref{def spec map} and the dual of the cospecialization map \Cref{def dual cosp} compare. 

We again let $\mathcal X\to B$ be a smooth family over a smooth curve $B$ with generic point $\eta$. We also let $B'\to B$ be a finite map from a smooth curve $B'$ with generic point $\eta'$. Write $\mathcal X':=\mathcal X\times_BB'$ for the base change. We start with a technical lemma.
\begin{lemma}\label{Cohomology Decomposition}
Let $b\in B$ be a closed point and suppose $\Het^i(\mathcal X,\Lambda)\to\Het^i(X_b,\Lambda)$ is surjective. Let $b'\in B'$ lying over $b$, then any $\beta'\in\Het^{i+1}(X_{\eta'},\mu_\ell)$ can be written as $\beta'=\tilde\beta'|_{X_{\eta'}}-(\pi_{b'})\cup\beta_0'$, for some $\tilde\beta'\in\Het^{i+1}(\mathcal X'_{U'},\Lambda)$ with $b'\in U'\subseteq B'$, $\beta_0'\in\Het^i(X_{\eta'},\Lambda)$ and $\pi_{b'}$ a uniformizer at $b'$.
\end{lemma}
\begin{remark}\label{deRham2}
To elaborate on the relation with logarithmic deRham cohomology of \Cref{deRham1}, in \cite[\textsection 4.1]{PS} to define the residue map one uses the fact that we can write $\omega$ as $\omega=\frac{\mathrm{d}t}{t}\wedge \eta+\eta'$, where $\eta'$ (and $\eta$) is global. The above lemma is analogous to this.
\end{remark}
\begin{proof}
First of all, note that the commutativity of the diagram
\[
\begin{tikzcd}
X'_{b'}\ar[r,"\cong"]\ar[d]&X_b\ar[d]\\
\mathcal X'\ar[r]&\mathcal X
\end{tikzcd}
\]
implies that $\Het^i(\mathcal X',\Lambda)\to\Het^i(X'_{b'},\Lambda)$ is also surjective. So without loss of generality, assume $B'=B$.

Let $\beta\in\Het^{i+1}(X_\eta,\Lambda)$, then by assumption, $Res_b\beta\in\Het^i(X_b,\Lambda)$ can be extended to $\widetilde{Res_b\beta}\in\Het^i(\mathcal X,\Lambda)$. Let $\beta_0:=\widetilde{Res_b\beta}|_{X_\eta}$, then using \Cref{global sp} we compute $Res_b(\beta+(\pi_b)\cup\beta_0)=Res_b\beta-sp_b(\beta_0)=Res_b\beta-\widetilde{Res_b\beta}|_{X_b}=0$, so there exists $\tilde\beta\in\Het^{i+1}(\mathcal X_U,\mu_\ell)$ with $b\in U\subseteq B$ such that $\tilde\beta|_{X_\eta}=\beta+(\pi_b)\cup\beta_0$, as wished.
\end{proof}

\begin{lemma}\label{completion of curve}
Let $f\colon(U,u)\to (B,b)$ be an irreducible étale neighborhood of $b\in B$ a smooth complete curve. Then there exists a smooth curve $B'$ with $U\subseteq B'$ and a finite morphism $B'\to B$ extending $U\to B$.
\end{lemma}
\begin{proof}
Let $B'$ be the normalized completion of $U$. Then we have a rational map $B'\to B$.
Using the valuative criterion of properness and the fact that $B$ is complete, we can extend over
every $b'\in B'\setminus U$ giving a map $B'\to B$ (\cite[0BXZ]{Stacks}). As $B'$ is smooth, complete and irreducible, \cite[Proposition II.6.8]{Hart} implies that $B'\to B$ is finite.
\end{proof}

\begin{prop}\label{general-comparison-cosp}
Same assumptions as in \Cref{Cohomology Decomposition}, but now assume that $\eta\subseteq\eta'$ is a finite separable extension (that is, $B'\to B$ is locally étale). Then the following diagram commutes
\[
\begin{tikzcd}
\Het^{i+1}(X_{\bar\eta},\Lambda)\ar[r,"cosp^\vee"]&\Het^{i+1}(X_b,\Lambda)\\
\Het^{i+1}(X_{\eta'},\Lambda)\ar[u]\ar[r,"sp_{b'}"]&\Het^{i+1}(X'_{b'},\Lambda)\ar[u,"\cong"]
\end{tikzcd}.
\]
\end{prop}
\begin{proof}
By \Cref{dualcosp-base-point-indep} it suffices to show the commutativity of 
\[
\begin{tikzcd}
\Het^{i+1}(X_{\bar{\eta}},\Lambda)\ar[r,"cosp^\vee"]&\Het^{i+1}(X'_{b'},\Lambda)\\
\Het^{i+1}(X_{\eta'},\Lambda)\ar[u]\ar[ru,"sp_{b'}"]&
\end{tikzcd}.
\]
For the sake of notation, write $\eta'=\eta$ and $b'=b$. Then we need the equality
\[
\int_{X_{\bar\eta}}cosp(\alpha)\cup\beta_{\bar\eta}=cosp\int_{X_b}\alpha\cup sp_b(\beta)
\]
to hold for every $\alpha\in\mrH_c^{2d-i-1}(X_b)$ and every $\beta\in\Het^{i+1}(X_\eta)$. As $\alpha\in (R^{2d-i-1}f_!\Lambda)_b$, there exists an étale neighborhood $(U,u)\to (B,b)$ so that $\alpha$ can be represented by some $\tilde\alpha\in\Gamma(U,R^{2d-i-1}f_!\Lambda)$ with $\tilde\alpha_u=\alpha$ and $\tilde\alpha_{\bar\eta}=cosp(\alpha)$.

We may assume that $U$ is irreducible with generic point $\eta'$, so we view $U\to B$ as an open of some finite  $B'\to B$ by \Cref{completion of curve}. Pull back $\beta$ to $\beta'\in\Het^{i+1}(X_{\eta'},\Lambda)$, then by \Cref{Cohomology Decomposition} we can write $\beta'=\tilde\beta'|_{X_{\eta'}}-(\pi_u)\cup\beta_0'$. Now we compute that
\[
\int_{X_{\bar\eta}}cosp(\alpha)\cup\beta'_{\bar\eta}=\int_{X_{\bar\eta'}}\tilde\alpha_{\bar\eta'}\cup\tilde\beta'|_{X_{\bar\eta}}=\int_{X_{\bar\eta}}(\tilde\alpha\cup\tilde\beta')_{\bar\eta}
\]
and
\[
cosp\int_{X_u}\alpha\cup sp_u(\beta')=cosp\int_{X_u}\tilde\alpha_u\cup\tilde\beta'|_{X_u}=cosp\int_{X_u}(\tilde\alpha\cup\tilde\beta')_u,
\]
which are equal by \Cref{constant trace}. This equality implies that the diagram
\[
\begin{tikzcd}
\Het^{i+1}(X_{\bar{\eta}},\Lambda)\ar[r,"cosp^\vee"]&\Het^{i+1}(X_u,\Lambda)\\
\Het^{i+1}(X_{\eta'},\Lambda)\ar[u]\ar[ru,"sp_u"]&
\end{tikzcd}
\]
commutes. By \Cref{compatibility-sp-fin-ext} we know that
\[
\begin{tikzcd}
\Het^{i+1}(X_{\eta'},\Lambda)\ar[r,"sp_u"]&\Het^{i+1}(X_u,\Lambda)\ar[d,"\cong"]\\
\Het^{i+1}(X_\eta,\Lambda)\ar[u]\ar[r,"sp_b"]&\Het^{i+1}(X_b,\Lambda)
\end{tikzcd}
\]
commutes. So again by \Cref{dualcosp-base-point-indep} we conclude the result.
\end{proof}

\subsection{Comparison for Lefschetz Pencils}\label{L pencil}
Here we are going to apply the previous section on a Lefschetz pencil $f\colon\mathcal X\to\P^1$ with fibres $2n-1$ dimensional projective hypersurfaces of degree $\geq 3$, \cite[III\textsection 1]{FK}. We let $0\in\P^1$ be a critical value and $X_0$ its corresponding singular fibre with ordinary double point $x_0\in X_0$. Note that now (locally around $0$) ${\mathcal X\setminus\{x_0\}\to\P^1}$ is a smooth family. We will then consider the maps $cosp^\vee\colon\Het^i(X_{\bar\eta},\Lambda)\to\Het^i(X_0\setminus\{x_0\},\Lambda)$ and $sp\colon \Het^i(X_\eta,\Lambda)\to\Het^i(X_0\setminus \{x_0\},\Lambda)$ and show they are compatible in the sense of the previous section.

\begin{prop}\label{extendable singular point}
Let $X\subseteq\P^{n+1}$ be a hypersurface of degree $\geq3$ with ordinary double point $x\in X$. Then the restriction map $\Het^{2r}(X,\Lambda)\to\Het^{2r}(X\setminus \{x\},\Lambda)$ is surjective for every $2r<n$.
\end{prop}
\begin{proof}
Let $\pi\colon\widetilde X:=\mathrm{Bl}_{\{x\}}(X)\to X$ be the blow-up with exceptional divisor $i\colon D\to\widetilde X$, which is a smooth quadric. Write $j\colon U=X\setminus\{x\}\cong\widetilde X\setminus D\to\widetilde X$ for the inclusion of the complement. The localization exact sequence gives us a commutative diagram
\[
\begin{tikzcd}
\Het^{2r-2}(D,\Lambda)\ar[r,"i_\ast"]&\Het^{2r}(\widetilde X,\Lambda)\ar[r, two heads,"j^\ast"]&\Het^{2r}(U,\Lambda)\ar[r]&\Het^{2r-1}(D,\Lambda)=0\\
&\Het^{2r}(X,\Lambda)\ar[u,"\pi^\ast"]\ar[r]&\Het^{2r}(X\setminus\{x\},\Lambda)\ar[u,"\cong"]&
\end{tikzcd}.
\]

To show that $j^\ast\circ\pi^\ast$ is surjective, it suffices to show that $\pi^\ast+i_\ast$ is surjective. For this, we can view $\widetilde X\subseteq\mathrm{Bl}_{\{x\}}(\P^{n+1}):=\widetilde\P$ as an embedded blow-up inside $\tau\colon\widetilde\P\to\P^{n+1}$, with exceptional divisor $l\colon\P^n\cong E\to\widetilde \P$. By \cite[Lemma 3.10]{Lindner} we know that $\widetilde X\subseteq\widetilde \P$ is an ample divisor, so that $\Het^i(\widetilde \P,\Lambda)\isomto\Het^i(\widetilde X,\Lambda)$ is an isomorphism for $i<n$.

We claim that the following diagram commutes
\[
\begin{tikzcd}
\Het^{2r}(\P^{n+1},\Lambda)\oplus\Het^{2r-2}(E,\Lambda)\ar[r,"\tau^\ast+l_\ast"]\ar[r,swap,"\cong"]\ar[d]&\Het^{2r}(\widetilde\P,\Lambda)\ar[d,"\cong"]\\
\Het^{2r}(X,\Lambda)\oplus\Het^{2r-2}(D,\Lambda)\ar[r,"\pi^\ast+i_\ast"]&\Het^{2r}(\widetilde X,\Lambda)
\end{tikzcd}.
\]
For this, write $k\colon\widetilde X\to\widetilde\P$, then as the cohomology of $E$ is generated by $j^\ast[E]^{r-1}$ we compute that $k^\ast j_\ast j^\ast[E]^{r-1}=k^\ast [E]^r=[D]^r$ and $i_\ast k^\ast j^\ast[E]^{r-1}=i_\ast i^\ast[D]^{r-1}=[D]^r$.

Now the top horizontal map is an isomorphism by the blow-up formula and the right vertical map is an isomorphism for $2r<n$ by the above. So we conclude that the bottom horizontal map is surjective, as wished.
\end{proof}

\begin{theorem}\label{Lefschetz Extendable}
Let $\mathcal X\to\P^1$ be a Lefschetz pencil as in the beginning of this section. Let $B'\to\P^1$ be a finite map and $\mathcal X':=\mathcal X\times_{\P^1}B'$ be the base change. Then $\Het^{2r}(\mathcal X,\Lambda)\to\Het^{2r}(X_0\setminus\{x_0\},\Lambda)$ is surjective for every $2r<\dim X_0$.
\end{theorem}
\begin{proof}
Suppose $\mathcal X$ is the Lefschetz pencil of hyperplane sections of some smooth projective space $Y$ (of the same dimension). Then $\mathcal X$ can be identified with the blow-up of $Y$ in the base locus, giving $\P^1\leftarrow\mathcal X\to Y$. By the Lefschetz hyperplane theorem, we have that $\Het^i(Y,\Lambda)\to\Het^i(X_0,\Lambda)$ is an isomorphism for $i<\dim X_0$. If the pencil consists of projective hypersurfaces, we conclude  from \Cref{extendable singular point} that $\Het^{2r}(X_0,\Lambda)\to\Het^{2r}(X_0\setminus\{x_0\},\Lambda)$ is surjective for $2r<\dim X_0$. Now note that
\[
\begin{tikzcd}
X_0\setminus\{x_0\}\ar[r]\ar[d]&X_0\ar[r]&Y\\
\mathcal X\ar[rru]&&
\end{tikzcd}
\]
commutes and as pull-back via the top horizontal composition is surjective, implies that pulling back via the vertical map is surjective.
\end{proof}

\begin{cor}\label{comparison-Lefschetz}
Let $\mathcal X\to\P^1$ be a Lefschetz pencil with odd dimensional fibres and $B'\to\P^1$ be a finite morphism from a smooth $B'$ with generic point $\eta'$, separable over $\eta$. Let $0'\in B'$ lying over $0\in \P^1$, then the diagram
\[
\begin{tikzcd}
\Het^i(X_{\bar\eta},\Lambda)\ar[r,"cosp^\vee"]&\Het^i(X_0\setminus\{x_0\},\Lambda)\\
\Het^i(X_{\eta'},\Lambda)\ar[u]\ar[r,"sp_{0'}"]&\Het^i(X'_{0'}\setminus\{x'_{0'}\},\Lambda)\ar[u,"(f_{0'})_\ast"]\ar[u,swap,"\cong"]
\end{tikzcd}
\]
commutes for every $i\leq\dim X_0$.
\end{cor}
\begin{proof}
Write $\mathcal X_{sm}$ for $\mathcal X$ with all the singular points of the singular fibres removed. The result follows if we can apply \Cref{general-comparison-cosp} to the smooth family $\mathcal X_{sm}\to\P^1$. This is possible if we show that $\Het^{i-1}(\mathcal X_{sm},\Lambda)\to\Het^{i-1}(X_0\setminus\{x_0\},\Lambda)$ is surjective for every $i\leq\dim(X_0)$.

When $i-1$ is even, this follows from \Cref{Lefschetz Extendable} as the surjection $\Het^{2r}(\mathcal X,\Lambda)\twoheadrightarrow\Het^{2r}(X_0\setminus\{x_0\},\Lambda)$ decomposes as $\Het^{2r}(\mathcal X,\Lambda)\to\Het^{2r}(\mathcal X_{sm},\Lambda)\to\Het^{2r}(X_0\setminus\{x_0\},\Lambda)$.

When $i-1$ is odd, we compute that $\Het^{i-1}(X_0\setminus\{x_0\},\Lambda)\cong\mrH_c^{2d-i+1}(X_0\setminus\{x_0\},\Lambda)^\vee\cong\Het^{2d-i+1}(X_0,\Lambda)^\vee=0$ by \cite[Theorem III.4.3(0)]{FK}, so \Cref{general-comparison-cosp}.
\end{proof}

\begin{remark} 
\Cref{comparison-Lefschetz} can be viewed more generally as follows, let $f\colon\mathcal X\to B$ be a proper degeneration over a smooth curve $B$. Then by \cite[0A3S]{Stacks} and the proof of \cite[095T]{Stacks} we get that $\Het^i(\mathcal X,\Lambda)\to\Het^i(X_b,\Lambda)$ is surjective for every closed point $b$. 

Let $b$ be the critical value and set $\mathcal X_{sm}:=\mathcal X\setminus Sing(X_b)$ so that $(X_b)_{sm}:=X_b\cap\mathcal X_{sm}$ and $\mathcal X_{sm}\to B$ is (around $b$) a smooth family. Then with the above, if $\Het^i(X_b,\Lambda)\to\Het^i((X_b)_{sm},\Lambda)$ is surjective, we have $\Het^i(\mathcal X_{sm},\Lambda)\to\Het^i((X_b)_{sm},\Lambda)$ is surjective. So as in \Cref{general-comparison-cosp}, the surjectivity of $\Het^i(X_b,\Lambda)\to\Het^i((X_b)_{sm},\Lambda)$ implies a comparison between $cosp^\vee$ and $sp$.
\end{remark}

\section{Application to the Integral Hodge/Tate Conjecture}\label{main section}
In this section we present our main result. We start with some computations of the cohomology groups of varieties related to an ordinary double point. After this we give the full computation of the cohomology of an Enriques surface. This is probably well known to the experts, but no precise reference could be found. After this, we obtain our main results directly.

\subsection{Residue of an Ordinary Double Point}

Let $X_0\subseteq\P^{2n}$ be a projective hypersurface of dimension $2n-1$ over an algebraically closed fied $k$ with ordinary double point $x_0\in X_0$ and of degree $\geq 3$. Let $\widetilde{X_0}:=\Bl_{\{x_0\}}(X_0)$ be the blow-up of $X_0$ in this singular point. Then $\widetilde{X_0}$ is smooth and the exceptional divisor is a smooth quadric $D\subseteq\widetilde{X_0}$ of dimension $2n-2$. We view $\widetilde{X_0}\subseteq\widetilde{\P^{2n}}$, where $\widetilde{\P^{2n}}:=\Bl_{\{x_0\}}(\P^{2n})$ with exceptional divisor $\P^{2n-1}\subseteq\widetilde{\P^{2n}}$ and $D\subseteq\P^{2n-1}$.

Let $\Lambda=\underline{\Lambda}$ be a constant sheaf with finite stalks of order prime to $\mathrm{char}(k)$. We note that $\widetilde{X_0}\setminus D\isomto X_0\setminus\{x_0\}$ is an isomorphism so that we have a residue map 
\[
Res\colon\Het^i(X_0\setminus\{x_0\},\Lambda)\isomto\Het^i(\widetilde X_0\setminus D,\Lambda)\to\Het^{i-1}(D,\Lambda).
\]

\begin{lemma}\label{rank 1}
We have $\Het^{2n}(\widetilde{X_0}\setminus D,\Lambda)\cong\begin{cases}\Lambda&\text{ if }n>1;\\ 0&\text{ if }n=1.\end{cases}$
\end{lemma}
\begin{proof}
Suppose $n=1$, then the localization sequence gives $0\to\mrH_c^0(X_0\setminus\{x_0\},\Lambda)\to\Het^0(X_0,\Lambda)\isomto\Het^0(x_0,\Lambda)$. So $\Het^2(\tilde X_0\setminus D,\Lambda)\cong\Het^2(X_0\setminus\{x_0\},\Lambda)\cong\mrH_c^0(X_0\setminus\{x_0\},\Lambda)^\vee=0$.

Now suppose $n>1$, then the Lefschetz Hyperplane Theorem \cite[Corollary I.9.4]{FK} gives that $\Lambda\cong\Het^{2n-2}(\P^{2n},\Lambda)\isomto\Het^{2n-2}(X_0,\Lambda)$. Using the localization exact sequence again now gives that $\mrH^{2n-2}_c(X_0\setminus\{x_0\},\Lambda)\isomto\Het^{2n-2}(X_0,\Lambda)$. So by Poincaré Duality we conclude that $\Het^{2n}(\widetilde X_0\setminus D,\Lambda)\cong \Het^{2n}(X_0\setminus\{x_0\},\Lambda)\cong \Hom_\Lambda(\mrH^{2n-2}_c(X_0\setminus\{x_0\},\Lambda),\Lambda)\cong\Lambda$.
\end{proof}

\begin{lemma}\label{rank 2}
We have $\Het^{2n}(\widetilde{X_0},\Lambda)\cong\begin{cases}\Lambda^{\oplus 2}&\text{ if }n>1;\\ \Lambda&\text{ if }n=1.\end{cases}$
\end{lemma}
\begin{proof}
When $n=1$ we see that $\Het^2(\widetilde{X_0},\Lambda)\cong\Het^0(\widetilde{X_0},\Lambda)^\vee\cong\Lambda$.

Suppose $n>1$, using Poincaré Duality again we have $\Het^{2n}(\widetilde{X_0},\Lambda)\cong\Hom_{\Lambda}(\Het^{2n-2}(\widetilde X_0,\Lambda),\Lambda)$, so it suffices to show that $\Het^{2n-2}(\widetilde{X_0},\Lambda)\cong\Lambda^{\oplus 2}$. As the degree of $X_0$ is at least 3, then by \cite[Lemma 3.10]{Lindner} $\widetilde{X_0}$ is an ample divisor of $\widetilde{\P^{2n}}$, so that $\Het^{2n-2}(\widetilde{\P^{2n}},\Lambda)\to\Het^{2n-2}(\widetilde{X_0},\Lambda)$ is an isomorphism. So the result follows from the blow-up formula.
\end{proof}

\begin{prop}\label{non-zero residue}
The residue map $\Het^{2n-1}(\widetilde{X_0}\setminus D,\Lambda)\to\Het^{2n-2}(D,\Lambda)$ is non-zero.
\end{prop}
\begin{proof}
Consider the long exact sequence
\[
\Het^{2n-1}(\widetilde{X_0}\setminus D,\Lambda)\overset{Res}{\to}\Het^{2n-2}(D,\Lambda)\to\Het^{2n}(\widetilde{X_0},\Lambda)\to\Het^{2n}(\widetilde{X_0}\setminus D,\Lambda)\to\Het^{2n-1}(D,\Lambda).
\]
Plugging in the computations from \Cref{rank 1} and \Cref{rank 2} and the cohomology of a smooth quadric (\cite[Proposition III.4.4]{FK}) gives
\begin{align*}
\Het^{2n-1}(\widetilde{X_0}\setminus D,\Lambda)\overset{Res}{\to}\Lambda^{\oplus 2}\to\Lambda\to0\to 0&\quad\text{ if } n=1;\\
\Het^{2n-1}(\widetilde{X_0}\setminus D,\Lambda)\overset{Res}{\to}\Lambda^{\oplus 2}\to\Lambda^{\oplus 2}\to\Lambda\to0&\quad\text{ if }n>1.
\end{align*}
So in the case $n=1$, the residue map clearly cannot be zero. If the residue map were zero in the second case, then the second map would be an isomorphism, implying that the last surjection is zero, a contradiction. \qedhere


\end{proof}

\begin{prop}\label{Existence Non-zero ResidueSp}
Let $\mathcal X\to B$ be a Lefschetz pencil with singular fibre $X_0$. There exists a finite extension $B'\to B$ and $\beta'\in\Het^{2n-1}(X_{\eta'},\Lambda)$ so that $Res(sp_{0'}(\beta'))\neq 0$ in $\Het^{2n-2}(D,\Lambda)$ for some $0'\in B'$ mapping to $0\in B$.
\end{prop}
\begin{proof}
By \cite[Theorem III.4.3(1)]{FK} we have that $cosp\colon\Het^{2n-1}(X_0,\Lambda)\to\Het^{2n-1}(X_{\bar\eta},\Lambda)$ is injective (with cokernel generated by the vanishing sphere). Using \Cref{relation cosp} this implies that $cosp\colon\mrH_c^{2n-1}(X_0\setminus \{x_0\},\Lambda)\isomto\Het^{2n-1}(X_0,\Lambda)\to\Het^{2n-1}(X_{\bar\eta},\Lambda)$ is injective. As $\Lambda$ is an injective $\Lambda$-module, the dual $cosp^\vee\colon\Het^{2n-1}(X_{\bar\eta},\Lambda)\to\Het^{2n-1}(X_0\setminus \{x_0\},\Lambda)$ is surjective (with kernel generated by the vanishing sphere). Thus by \Cref{non-zero residue}, there exists a $\beta\in\Het^{2n-1}(X_{\bar\eta},\Lambda)$ such that $cosp^\vee(\beta)$ has non-zero residue.

Let $\beta'\in\Het^{2n-1}(X_{\eta'},\Lambda)$ representing $\beta$ for some finite (separable) extension $\eta'\to\eta$. Then by \Cref{comparison-Lefschetz} we conclude that $sp_{0'}(\beta')=cosp^\vee(\beta)$ has non-zero residue.
\end{proof}

\subsection{Non-Vanishing Refined Unramified Cohomology}

We start with the computation of the cohomology of an Enriques surface.
\begin{prop}\label{cohom-Enriques}
Let $S$ be an Enriques surface over an algebraically closed field $k$ and let $\ell\neq \mathrm{char}(k)$ be a prime. For $\ell\neq 2$ we have
\begin{align*}
\Het^0(S,\mu_{\ell^r})&\cong\Z/\ell^r\Z\cong\Het^4(S,\mu_{\ell^r}) & \Het^0(S,\Z_\ell)&\cong\Z_\ell\cong\Het^4(S,\Z_\ell);\\
\Het^1(S,\mu_{\ell^r})&=0=\Het^3(S,\mu_{\ell^r}) & \Het^1(S,\Z_\ell)&=0=\Het^3(S,\Z_\ell);\\
\Het^2(S,\mu_{\ell^r})&\cong(\Z/\ell^r\Z)^{\oplus 10} & \Het^2(S,\Z_\ell)&\cong\Z_\ell^{\oplus 10}.
\end{align*}
If $\ell=2$, we have
\begin{align*}
\Het^0(S,\mu_{2^r})&\cong\Z/2^r\Z\cong\Het^4(S,\mu_{2^r}) & \Het^0(S,\Z_2)&\cong\Z_2\cong\Het^4(S,\Z_2);\\
\Het^1(S,\mu_{2^r})&=\Z/2\Z & \Het^1(S,\Z_2)&=0;\\
\Het^2(S,\mu_{2^r})&\cong(\Z/2^r\Z)^{\oplus 10}\oplus(\Z/2\Z)^{\oplus 2} & \Het^2(S,\Z_2)&\cong\Z_2^{\oplus 10}\oplus\Z/2\Z;\\
\Het^3(S,\mu_{2^r})&=\Z/2\Z & \Het^3(S,\Z_2)&=\Z/2\Z.
\end{align*}
\end{prop}
\begin{proof}
By the assumptions on $\ell$ and $k$, we see that $\Het^0(S,\mu_{\ell^r})=\Gamma(S,\mu_{\ell^r})\cong\Z/\ell^r\Z$ and $\int_S\colon\Het^4(S,\mu_{\ell^r})\isomto\Z/\ell^r\Z$. From this, the $\Z_\ell$ coefficients also follow.

We will use the Kummer sequence $0\to\mu_{\ell^r}\to\G_m\overset{(-)^{\ell^r}}{\to}\G_m\to 0$ to compute the remaining cohomological degrees. By assumption on $k$, the short exact Kummer sequence induces the exact sequences
\[
0\to \Pic(S)/\ell^r \Pic(S)\to\Het^2(S,\mu_{\ell^r})\to\Het^2(S,\G_m)[\ell^r]\to0
\]
and
\[
0\to\Het^1(S,\mu_{\ell^r})\to\Pic(S)[\ell^r]\to0.
\]

First we claim that $\Pic(S)\cong\Z^{\oplus 10}\oplus\Z/2\Z$, which gives the torsion cohomology in degree 1, and in turn degree 3 via Poincaré Duality. For the claim, use \cite[Theorem 1.2.7]{Enr-Surf} giving that $rk(NS(S))=\rho(S)=10$, where $NS(S)=\Pic(S)/\Pic^0(S)$. And Theorem 1.2.1 of loc. cit. gives $\Pic^0(S)=0$, hence we conclude that $\Pic(S)$ is of rank 10. Moreover, in the proof of Theorem 1.2.1 it is given that $Tors(\Pic(S))=\Z/2\Z$, generated by the canonical divisor, proving the claim.

By \cite[Theorem 1.2.17]{Enr-Surf}, we have that $Br(S)\cong\Z/2\Z$, implying that $\Het^2(S,\G_m)[\ell^r]=\begin{cases}0 &\text{ if }\ell\neq 2\\ \Z/2\Z&\text{ if } \ell=2\end{cases}$. So if $\ell\neq 2$, we have a full description of the torsion cohomology. Also, combining the commutative diagram
\[
\begin{tikzcd}
0\ar[r]&\mu_{\ell^r}\ar[r]\ar[d,"(-)^\ell"]&\G_m\ar[r,"(-)^{\ell^r}"]\ar[d,"(-)^\ell"]&\G_m\ar[r]\ar[d,"\id"]&0\\
0\ar[r]&\mu_{\ell^{r-1}}\ar[r]&\G_m\ar[r,"(-)^{\ell^{r-1}}"]&\G_m\ar[r]&0\\
\end{tikzcd}
\]
with the long exact sequences gives the transition morphisms in this case, computing the $\ell$-adic cohomology.

For $\ell=2$ we obtain the short exact sequence
\[
0\to(\Z/2^r\Z)^{\oplus 10}\oplus\Z/2\Z\to\Het^2(S,\mu_{2^r})\to\Z/2\Z\to 0,
\]
and we claim that this sequence splits for every $r$. Clearly it does for $r=1$, so we assume $r\geq 2$. As $\mathrm{Ext}^1_{\Z/2^r\Z}((\Z/2^r\Z)^{\oplus 10}\oplus \Z/2\Z,\Z/2\Z)\cong\Z/2\Z$, there are only two options, either the sequence splits, or the middle term equals $(\Z/2^r\Z)^{\oplus 10}\oplus\Z/4\Z$. Suppose for some $r$ the term is $\Z/4\Z$, then using the above commutative diagram, we would then have a commutative diagram
\[
\begin{tikzcd}
0\ar[r]&\Z/2\Z\ar[r]\ar[d,"\id"]&\Z/4\Z\ar[r]\ar[d,"f"]&\Z/2\Z\ar[r]\ar[d,"0"]&0\\
0\ar[r]&\Z/2\Z\ar[r]&A\ar[r]&\Z/2\Z\ar[r]&0\\
\end{tikzcd},
\]
where $A\in\{(\Z/2\Z)^{\oplus2},\Z/4\Z\}$. But one can check that for both options for $A$, such an $f$ making the diagram commute does not exist. So we conclude that the sequence must split for every $r$ and thus $\Het^2(S,\mu_{2^r})\cong(\Z/2^r\Z)^{\oplus 10}\oplus(\Z/2\Z)^{\oplus 2}$. And in turn $\Het^2(S,\Z_2)\cong\Z_2^{\oplus 10}\oplus\Z/2\Z$.

The first paragraph implies that $\Het^1(S,\Z_2)=0$. For $\Het^3(S,\Z_2)$ we use the long exact Bockstein sequence $0\to\mu_{2}\to\mu_{2^r}\overset{(-)^2}{\to}\mu_{2^{r-1}}\to0$ to conclude that $\Het^3(S,\mu_{2^r})\cong\Z/2\Z\to\Z/2\Z\cong\Het^3(S,\mu_{2^{r-1}})$ is the identity.
\end{proof}

\begin{cor}\label{cohom-class-Enriques}
Let $S$ be an Enriques surface over an algebraically closed field $k$ with $\mathrm{char}(k)\neq 2$. Then there exists an element $\alpha\in\Het^2(S,\mu_2)$ that does not lift to $\mu_4$ coefficients. In particular $k(\alpha)\neq0$ in $\Het^2(k(S),\mu_2)$.
\end{cor}
\begin{proof}
From the proof of \Cref{cohom-Enriques}, we see that we can pick $\alpha$ to be the non-zero class coming from $\Het^2(S,\G_m)[2]$ as the map $\Het^2(S,\G_m)[4]\to\Het^2(S,\G_m)[2]$ is zero.

If $k(\alpha)$ were zero, then $\alpha$ is algebraic by \Cref{les Sch}, impying it would lift to any $\Het^2(S,\mu_{\ell^r})$.
\end{proof}

\begin{lemma}\label{Product Algebraic and Liftable}
Let $X$ and $Y$ be smooth projective varieties. 
\begin{enumerate}
\item Suppose the cohomology of $X$ is algebraic. Then for any $0\neq\alpha\in\Het^{2n}(X,\mu_\ell)$ and $\beta\in\Het^{2m}(Y,\mu_\ell)$ we have $\alpha\times\beta$ is algbraic if and only if $\beta$ is algebraic.

\item Suppose that $\Het^\ast(X,\mu_{\ell^r})\twoheadrightarrow\Het^\ast(X,\mu_\ell)$ is surjective and let $\alpha\in\Het^i(X,\mu_\ell)$ and $\beta\in\Het^j(Y,\mu_\ell)$. Then $\alpha\times\beta$ lifts to $\mu_{\ell^r}$ if and only if $\beta$ does.
\end{enumerate}
\end{lemma}
\begin{proof}
\emph{1.} As $\alpha$ is algebraic by assumption, the `if'-part is clear.

Conversely, write $\alpha\times\beta=[\Gamma]$ for some algebraic cycle $\Gamma$. By assumption, we know $\alpha=[Z]$ is algebraic as well. Let $\alpha'\in\Het^{2(d_X-n)}(X,\mu_\ell)$ so that $\int_X\alpha\cup\alpha'=\lambda$ is non-zero, which exists because $\Het^{2n}(X,\mu_\ell)\cong\Het^{2(d_X-n)}(X,\mu_\ell)^\vee$ where $\alpha\mapsto\alpha^\vee(-):=\int_X\alpha\cup-$, which is non-zero by assumption. By the assumption on $X$, there exists a closed $Z'\in \CH^{d_x-n}(X)$ so that $\alpha'=[Z']$. Then we compute that $\lambda\beta=(\alpha\times\beta)^\ast\alpha'=([\Gamma])^\ast([Z'])$ is algebraic, hence $\beta$ is because $\lambda$ is invertible.

\emph{2.} As $\alpha$ lifts to $\mu_{\ell^r}$ by assumption, the `if'-part is clear.

Conversely, if we let $\alpha'$ be as before, then by assumption, it lifts to $\mu_{\ell^r}$ so we again conclude that $\lambda\beta=(\alpha\times\beta)^\ast\alpha'$ lifts to $\mu_{\ell^r}$ hence $\beta$ does as the lift of $\lambda$ in $\mu_{\ell^r}$ remains invertible.
\end{proof}

\begin{theorem}\label{failure ITC}
Let $\alpha\in\Het^2(S,\mu_2)$ be the one from \Cref{cohom-class-Enriques} and let $\mathcal X\to\P^1$ be a Lefschetz pencil as in \Cref{L pencil}. There exists a $\beta\in\Het^{2n-1}(X_{\bar\eta},\mu_2)$ such that $F_{n-1}(\beta\times\alpha)\in\frac{\Hrnr{n-1}^{2n+1}(X_{\bar\eta}\times S,\mu_2)}{\Hrnr{n-1}^{2n+1}(X_{\bar\eta}\times S,\Z_2)}\cong Z^{n+1}(X_{\bar\eta}\times S)[2]$ is non-zero.
\end{theorem}
\begin{proof}
Let $\beta\in\Het^{2n-1}(X_{\bar\eta},\mu_2)$ be the image of the $\beta'$ from \Cref{Existence Non-zero ResidueSp} in $\Het^{2n-1}(X_{\bar\eta},\mu_2)$. And suppose $F_{n-1}(\beta\times\alpha)=0$ in $\Het^{2n+1}(F_{n-1}(X_{\bar\eta}\times S),\mu_2)$. By \Cref{refined-geometric-point} there now exists a finite extension $\eta''\to\eta'$ and $\beta'':=(\beta')_{\eta''}\in\Het^{2n-1}(X_{\eta''},\mu_2)$ so that $F_{n-1}(\beta''\times\alpha)=0\in\Het^{2n+1}(F_{n-1}(X_{\eta''}\times S),\mu_2)$. Suppose $B''\to B'$ is the finite map of smooth curves extending $\eta''\to\eta'$ with some $0''\mapsto 0'$. Then using \Cref{Specialization Refined} and \Cref{product-sp} we see that $F_{n-1}(sp_{0''}(\beta'')\times\alpha)=sp_{0''}F_{n-1}(\beta''\times\alpha)=0$. Taking the residue, we conclude by \Cref{Residue Refined} and \Cref{product-residue} that $F_{n-1}(\gamma\times\alpha)=0$ in $\Het^{2n}(F_{n-1}(D\times S),\mu_2)$, where $\gamma:=Res(sp_{0''}(\beta''))$. By \Cref{les Sch} this happens if and only if $\gamma\times\alpha$ is algebraic. Note that \Cref{compatibility-sp-fin-ext} gives $(f_{0''})_\ast sp_{0''}(\beta'')=sp_{0'}(\beta')$, so by definition of $\beta'$, we conclude that $\gamma\neq 0$. Now \Cref{Product Algebraic and Liftable} gives that $\alpha$ is algebraic, but this would mean that $k(\alpha)=0$, a contradiction.

Now suppose $F_{n-1}(\beta\times\alpha)\in\Hrnr{n-1}^{2n+1}(X_{\bar\eta}\times S,\mu_2)$ lifts to $\Hrnr{n-1}^{2n+1}(X_{\bar\eta}\times S,\Z_2)$, then in particular it lifts to $\Hrnr{n-1}^{2n+1}(X_{\bar\eta}\times S,\mu_4)$. That is, there exists a $\delta\in \mrH^{2n+1}(F_n(X_{\bar\eta}\times S),\mu_4)$ so that $F_{n-1}(\delta)\mod 2=F_{n-1}(\beta\times\alpha)\in\Hrnr{n-1}^{2n+1}(X_{\bar\eta}\times S,\mu_2)$. This implies that there is a finite extension $\eta''\to\eta'$ and $\delta''\in \mrH^{2n+1}(F_n(X_{\eta''}\times S),\mu_4)$ with the property that $F_{n-1}(\delta'')\mod 2=F_{n-1}(\beta''\times\alpha)\in\mrH^{2n+1}(F_{n-1}(X_{\eta''}\times S),\mu_2)$. As before, now also using \Cref{coeff sp}, this implies that $F_{n-1}(\gamma\times\alpha)=0\in\frac{\Hrnr{n-1}^{2n}(D\times S,\mu_2)}{\Hrnr{n-1}^{2n}(D\times S,\mu_4)}$. Using \Cref{liftable lemma} this implies that $\gamma\times\alpha$ lifts to $\mu_4$, hence $\alpha$ would lift to $\mu_4$ by \Cref{Product Algebraic and Liftable}, a contradiction  again by our choice of $\alpha$ from \Cref{cohom-class-Enriques}.
\end{proof}

\begin{remark}\label{explicite class}
Now that we know that $X_{\bar\eta}\times S$ contains non-algebraic torsion cohomology classes, it is interesting to see which class it exactly is. From the proof of \cite[Theorem 7.7]{Sch1} we see that the non-algebraic class is given by $\delta(\beta\times\alpha)\in\Het^{2n+2}(X_{\bar\eta}\times S,\Z_2(n+1))[2]$, where $\delta$ is the Bockstein homomorphism. As the cohomology of $X_{\bar\eta}$ is torsion free, we see that $\delta(\beta\times\alpha)=\beta'\times\delta(\alpha)$, where $\beta'\in\Het^{2n-1}(X,\Z_2)$ is so that $\beta'\mod 2=\beta$, in particular, $\beta'$ is not divisible by $2$. Moreover, $\delta(\alpha)\in\Het^3(S,\Z_2)\cong\Z/2\Z$ is the unique non-zero class. So when $k=\C$, we indeed recover the same non-algebraic class found by Shen in \cite{S}.
\end{remark}

\subsection{General Statement}\label{general statement}
Let $\mathcal X\to\P^1$ be a Lefschetz pencil as before and write $X:=X_{\bar\eta}$ for its geometric generic fibre. Using the exact same arguments as in the proof of \Cref{failure ITC} we obtain the following.

\begin{theorem}\label{cor-gen-IHC}
Let $Y$ be a smooth variety, then the following hold
\begin{enumerate}[(i)]
\item if $\Hrnr{i-1}^{2i}(Y,\mu_\ell)\neq 0$, then $\Hrnr{n+i-2}^{2n+2i-1}(X\times Y,\mu_\ell)\neq0$;\\
\item if $\frac{\Het^{2i}(Y,\mu_\ell)}{\Het^{2i}(Y,\mu_{\ell^r})}\overset{\ref{liftable lemma}}{\cong}\frac{\Hrnr{i-1}^{2i}(Y,\mu_\ell)}{\Hrnr{i-1}^{2i}(Y,\mu_{\ell^r})}\neq0$ for some $r$, then $Z^{n+i}(X\times Y)[\ell]\neq 0$.
\end{enumerate}
\end{theorem}

\begin{cor}\label{gen failure ITC}
Let $S$ be an Enriques surface. Then $Z^{n+i}(X\times S^i)[2]\neq 0$ for every $i\in\Z_{\geq 1}$.
\end{cor}
\begin{proof}
Let $i\in\Z_{\geq 1}$ and write $Y=S^i$. By \Cref{cor-gen-IHC}, it suffices to show there exists an $\alpha\in\mrH^{2i}(Y,\mu_\ell)$ that does not lift to $\mu_{\ell^r}$ for some $r$.

\Cref{cohom-class-Enriques} gives the result for $i=1$. Suppose inductively that $\alpha\in\Het^{2i-2}(S^{i-1},\mu_2)$ is a non-extendable class, say the extendability fails for $\mu_{2^r}$. Let $\beta\in\Het^2(S,\mu_{2^r})\cong (\Z/2^r\Z)^{10}\oplus(\Z/2\Z)^2$ be not a multiple of 2,  so that there exists $\beta'\in\Het^2(S,\mu_{2^r})\cong\Hom(\Het^2(S,\mu_{2^r}),\Z/2^r\Z)$ with the property $\beta\cup\beta'=1$. As this implies that $\bar\beta\cup\bar\beta'=1$, we have $(\alpha\times\bar\beta)^\ast\bar\beta'=\alpha$, where we write $\overline{(-)}$ for the natural map on cohomology to $\mu_2$ coefficients. So if $\alpha\times\bar\beta$ is extendable, then $\alpha$ is as well. So we conclude that $\alpha\times\bar\beta\in\Het^{2i}(S^i,\mu_2)$ cannot be lifted to $\mu_{2^r}$. (Note that $r=2$ suffices here.)
\end{proof}

\subsection{Variety of Lines}
Let $X\subseteq\P^{n+1}$ be a smooth cubic hypersurface over $k=\bar k$ with $n$ odd. Let $F=F(X)\subseteq Gr(\P^{n+1},\P^1)$ be the variety of lines, which is known to be smooth of dimension $2n-4$, \cite[Corollary 1.4 + Corollary 1.12]{AK}. We also write $P:=\{(x,[\ell])\in X\times F\mid x\in\ell\}$ for the incidence subscheme of $X\times F$. Note that $P\to F$ is a (smooth) $\P^1$-bundle over $F$ and $P$ is of codimension $n-1$ in $X\times F$. The action by correspondence of $P$ is known as the cylinder homomorphism $\Phi:=[P]^\ast\colon \Het^{3n-6}(F,\Z/\ell^r\Z)\to \Het^n(X,\Z/\ell^r\Z)$. We can take its inverse limit over $r$ to obtain $\Phi\colon \Het^{3n-6}(F,\Z_\ell)\to\Het^n(X,\Z_\ell)$. The following is a direct generalization of \cite[Lemma 3.2]{Renjie} (which comes from \cite[Lemma 3.17]{Renjiethesis}).
\begin{prop}\label{cylinder isom}
The cylinder homomorphism $\Phi\colon\Het^{3n-6}(F,\Z_\ell)\to\Het^n(X,\Z_\ell)$ is an isomorphism.
\end{prop}
\begin{proof}
Suppose first that the transcendence degree of $k$ over its prime field is less than the cardinality of $\C$.

Suppose $\mathrm{char}(k)=0$ so we can embed $k\subseteq \C$, then using \cite[Corollary VI.4.3]{Milne} we have natural isomorphisms induced by the pullback $\Het^i(F,\mu_{\ell^r})\isomto\Het^i(F_\C,\mu_{\ell^r})$ and $\Het^i(X,\mu_{\ell^r})\isomto\Het^i(X_\C,\mu_{\ell^r})$ for every $r$. So it suffices to show $\Het^{3n-6}(F_\C,\mu_{\ell^r})\to\Het^n(X_\C,\mu_{\ell^r})$ is an isomorphism for every $r$.

As the cohomology of $X$ and $F$ are torsion free, \cite[\textsection 3.1]{S} and \cite[Corollary 4.14]{Renjiethesis} respectively, we conclude by \cite{Shimada} that for $X$ and $F=F(X)$ over $\C$, the cylinder morphism $\Het^{3n-6}(F,\Z)\to\Het^n(X,\Z)$ is an isomorphism.  Moreover, $\Het^{3n-6}(F,\Z/m\Z)\to\Het^n(X,\Z/m\Z)$ is an isomorphism for every $m$. So the comparison theorem above implies that $\Het^{3n-6}(F_\C,\mu_{\ell^r})\to\Het^n(X_\C,\mu_{\ell^r})$ is an isomorphism for every $r$.

Suppose $\mathrm{char}(k)=p$, then let $R=W_p(k)$ the Witt ring of mixed characteristic. As is done above \cite[Lemma 3.17]{Renjiethesis} we can lift this smooth projective hypersurface $X$ to $\mathcal X\to R$, with generic point $Q=\mathrm{Frac}(R)$ of characteristic 0 and special fibre $X$.

Now we can embed $Q\subseteq\C$. And thus have a geometric point $\bar Q\colon\spec(\C)\to\spec(Q)\to\spec(R)$. So by the smooth and proper base change \cite[Corollary 4.2]{Milne}, it suffices to show $\Het^{3n-6}(F_{\bar Q},\Z_\ell)\to\Het^n(X_{\bar Q},\Z_\ell)$ is an isomorphism. Which follows if we can show that $\Het^{3n-6}(F_{\bar Q},\mu_{\ell^r})\to\Het^n(X_{\bar Q},\mu_{\ell^r})$ is an isomorphism for every $r$. This we have seen above.

For a general field $k$ (without cardinality assumption), because we are working with varieties, we can always find a subfield $k_0\subseteq k$ and $X_0/k_0$ so that $X_0\times_{k_0}k=X$ and $\overline{k_0}$ is small as assumed in the beginning of the proof. Now we can do the argument above, working with $\overline{X_{0}}=X_0\times_{k_0}\bar{k_0}$ and using again \cite[Cor VI.4.3]{Milne} we see that $\Het^i(\overline{X_{0}},\mu_{\ell^r})\cong\Het^i(X,\mu_{\ell^r})$, giving the result for general fields $k$.
\end{proof}

Now let $X$ be as in \Cref{general statement} and $F=F(X)$.
\begin{cor}\label{var of lines failure}
For every $i\in\Z_{\geq 1}$ we have $Z^{\frac{3n-5}{2}+i}(F\times S^i)[2]\neq 0$.
\end{cor}
\begin{proof}
We know by \Cref{gen failure ITC} that $X\times S^i$ contains non-algebraic cohomology classes. Similar to \Cref{explicite class} we see that this class is of the form $\beta\times\alpha\in\Het^n(X,\Z_2)\otimes\Het^{2i+1}(S^i,\Z_2)$.

\Cref{cylinder isom} gives an algebraic correspondence isomorphism $\Het^{3n-6}(F,\Z_2)\otimes\Het^{2i+1}(S^i,\Z_2)\overset{\Phi\times\Delta}{\to}\Het^n(X,\Z_2)\otimes\Het^{2i+1}(S^i,\Z_2)$. As its image contains non-algebraic (torsion) classes by the above, then so must $\Het^{3n-5+2i}(F\times S,\Z_2)$.
\end{proof}

\section{Closing Remarks}\label{closing remarks}

\subsection{Very General Fibre}
Given a family $\mathcal X\to S$ over a universal domain $\Omega$, then by \cite[Lemma 2.1]{Vial} we know that the geometric generic fibre is isomorphic to a very general fibre.

Now suppose $\mathcal X\to S$ is defined over any algebraically closed field $k$. Let $k\subseteq\Omega$ be any inclusion into a universal domain. Consider the base change $\mathcal X_\Omega\to S_\Omega$ of $X\to S$ to $\Omega$. By the above we know that there exists a very general fibre $s_\omega\in S_\Omega(\Omega)$ such that $X_{\overline{\Omega(S_\Omega)}}\cong X_{s_\omega}$. Say $s_\omega\mapsto s\in S(k)$, so $s_\omega\subseteq s$ can be seen as a field extension of separably closed fields. Then we know that the map $\Het^\ast(X_s,\Z_\ell)\to\Het^\ast(X_{s_\omega},\Z_\ell)$, the pullback along $X_{s_\omega}\to X_s$, is an isomorphism \cite[0DDG]{Stacks} or \cite[Corollary VI.4.3]{Milne}. The pullback map preserves algebraic cycles, so suppose we have found a non-algebraic class in $\Het^\ast(X_{\overline{\Omega(S)}},\Z_\ell)\cong\Het^\ast(X_{s_\omega},\Z_\ell)$ (eg. \Cref{failure ITC}), then there must also be a non-algebraic class in $\Het^\ast(X_s,\Z_\ell)$, on a very general fibre of $\mathcal X\to S$.

\subsection{The Result by Colliot-Thélène}
Here we give a direct generalization of the results of \cite{CT} on which the present paper is based.
\begin{prop}[{cf. \cite[Proposition A7]{Gabber}}]\label{Gabber Refined}
Let $X\to S$ be a smooth morphism with $X$ a finite dimensional smooth irreducible variety over a field $k$  with qcqs fibres and $S$ a smooth curve. Let $\ell$ be a prime invertible in $k$ and $\alpha\in \mrH^i(X,\mu_{\ell^r})$. Then the set $\Phi:=\{s\in S\mid F_j(\alpha|_{X_{\bar s}})=0\in\mrH^i(F_j (X_{\bar s}),\mu_{\ell^r})\}$ is closed under specialization.
\end{prop}
\begin{proof}

Let $\eta$ be the generic point of $S$. We want to show if $\eta\in\Phi$, then $s\in\Phi$ for every closed point $s\in S$.

So suppose $F_j(\alpha|_{X_{\bar\eta}})=0$, using \Cref{refined-geometric-point} there exists a finite extension $\eta'\to\eta$ such that $F_j(\alpha|_{X_{\eta'}})=0$. So by \cite[0BXX]{Stacks} there exists a finite morphism $f\colon S'\to S$ where $\eta'$ is the generic point of a smooth curve $S'$. We can find $s'\in S'$ such that $f(s')=s$ and $\kappa(s)\to\kappa(s')$ is a finite extension. We let $R:=\mathcal O_{S',s'}$, which is a DVR as $s'\in S'$ is a smooth point. 


The map $\spec(R)\to S'$ hits precisely $\eta'$ and $s'$. Consider the base change $X_R:=X_{S'}\times_{S'}R$, which has fibres $X_{\eta'}$ and $X_{s'}$. Then using \Cref{global sp} and \Cref{Specialization Refined} we see that $F_j(\alpha|_{X_{s'}})=sp(F_j(\alpha|_{X_{\eta'}}))=0$. Using again \Cref{refined-geometric-point} we conclude that $F_j(\alpha|_{X_{\overline{s}}})=0$ and thus $s\in \Phi$ as wished.
\end{proof}

\begin{theorem}[{cf. \cite[Théorème 1.1]{CT}}]\label{CT Refined}
Let $\ell$ be a prime and suppose that the natural map $\mrH^i(X,\mu_{\ell})\to\mrH^i(F_jX,\mu_{\ell})$ is non-zero. Then there exists an elliptic curve $E$ such that the natural map $\mrH^{i+1}(X\times E,\mu_{\ell})\to\mrH^{i+1}(F_j(X\times E),\mu_{\ell})$ is non-zero. 
\end{theorem}
\begin{proof}
As in the proof of \cite[Théorème 1.1]{CT} a result of Gabber \cite[Proposition A4]{Gabber} gives a class $\beta\in\mrH^1(Y,\mu_\ell)$, where $Y=\mathcal E\to U$ is a family with general fibres elliptic curves and special fibre $\mathbb G_m$ over a $P\in U(\Q)$, such that $Res(\beta|_P)\neq 0\in \mrH^0(\{0\},\mu_\ell)$, where we view $\mathbb G_m=\A^1\setminus\{0\}$.

Let $\alpha\in\mrH^i(X,\mu_\ell)$ such that $F_j(\alpha)\neq0$. Then $\alpha\times\beta\in\mrH^{i+1}(X\times Y_\C,\mu_\ell)$ and $Res(\alpha\times\beta|_P)=\pm\alpha\times Res(\beta|_P)=u\cdot\alpha\in\mrH^i(X,\mu_\ell)$, where $u\in\mu_\ell$ is non-zero. Thus $F_j(Res(\alpha\times\beta|_P))=u\cdot F_j(\alpha)$ is non-zero in $\mrH^i(F_jX,\mu_\ell)$ by assumption. By \Cref{Residue Refined}, we have that $Res(F_j(\alpha\times\beta|_P))\neq 0$ and thus in particular that $F_j(\alpha\times\beta|_P)\neq 0$. 

By \Cref{Gabber Refined}, the set $\Phi=\{u\in U\mid F_j(\alpha\times\beta|_{\bar u})=0\}$ is closed under specialization. The above implies that $P\not\in\Phi$, so in particular the $\eta_U\not\in\Phi$. Because the geometric generic fibre can be identified with a very general fibre \cite[Lemma 2.1]{Vial}, we conclude that there exists a very general elliptic curve $E$ such that $\mrH^{i+1}(X\times E,\mu_{\ell^r})\to\mrH^{i+1}(F_j(X\times E),\mu_{\ell^r})$ is non-zero.
\end{proof}

\begin{cor}[cf. \Cref{cor-gen-IHC}]
Let $X/\C$ be a smooth variety and $\alpha\in\mrH^{2n}(X,\mu_\ell)$ such that $F_{n-1}(\alpha)\neq0\in\mrH^{2n}(F_{n-1}X,\mu_\ell)$ and $\alpha$ does not lift to $\mu_{\ell^r}$ for some $r$. Then there exists an elliptic curve $E$ such that the integral Hodge conjecture for $n+1$ cycles fails on $X\times E$.
\end{cor}
\begin{proof}
We know that $Z^{n+1}(X\times E)[\ell]\cong\frac{\Hrnr{n-1}^{2n+1}(X\times E,\mu_\ell)}{\Hrnr{n-1}^{2n+1}(X\times E,\Z)}$ and by \Cref{CT Refined} we have a non-zero class $F_{n-1}(\alpha\times\beta)\in\Hrnr{n-1}^{2n+1}(X\times E,\mu_\ell)$. By \Cref{refined-geometric-point} we know that $F_{n-1}(\alpha\times\beta)$ lifts to $\Z$ if and only if there exists a finite extenstion $\eta'\to\eta$ such that $F_{n-1}(\alpha\times\beta_{\eta'})$ lifts to $\Z$. Thus $F_{n-1}(\alpha\times\beta_{\eta'})$ lifts to $\mu_{\ell^r}$. This implies that $F_{n-1}(Res(sp(\alpha\times\beta_{\eta'})))$ lifts to $\mu_{\ell^r}$. By \Cref{liftable lemma} this implies that $Res(sp(\alpha\times\beta_{\eta'}))$ lifts to $\mu_{\ell^r}$.  But $Res(sp(\alpha\times\beta_{\eta'}))=u\times\alpha$ with $u$ a unit hence this implies that $\alpha$ lifts to $\mu_{\ell^r}$, a contradiction.

We conclude that $F_{n-1}(\alpha\times\beta)$ is non-zero in the quotient $\frac{\Hrnr{n-1}^{2n+1}(X\times E,\mu_\ell)}{\Hrnr{n-1}^{2n+1}(X\times E,\Z)}\cong Z^{n+1}(E\times X)[\ell]$.
\end{proof}

\subsection{Refined Cospecialization}


In this section we show that the dual of the cospecialization map \Cref{def dual cosp} can be defined compatibly on the refinement degrees. With this, the application on the integral Hodge/Tate conjecture (eg. \Cref{failure ITC}) could have been concluded, without making use of the comparison between the specialization and the cospecialization map \Cref{comparison-Lefschetz} and just use the cohomology class $\beta$ found in the proof of \Cref{Existence Non-zero ResidueSp} with respect to $cosp^\vee$.

\begin{prop}
Let $\mathcal X\to B$ be a smooth family over a curve and $b\in B$ a closed point. For every $j\geq0$ there is a map $\Het^i(F_jX_{\bar\eta},\Lambda)\to\Het^i(F_j X_b,\Lambda)$ compatible with $F_j\to F_{j-1}$ and equal to $cosp^\vee\colon\Het^i(X_{\bar\eta},\Lambda)\to\Het^i(X_b,\Lambda)$ for $j\gg0$.
\end{prop}
\begin{proof}
As expected, the proof has a similar flavour as \cite[\textsection 4.2]{Sch2}.

Let $\alpha\in\Het^i(F_jX_{\bar\eta},\Lambda)$, so there exists $F_jX_{\bar\eta}\subseteq U\subseteq X_{\bar\eta}$ with complement $Z$ so that $\codim_{X_{\bar\eta}}(Z)>j$ and $\alpha$ is represented by some $\alpha_U\in\Het^i(U,\Lambda)$. There exists a finite extension $\eta'\to\eta$ and $Z'\subseteq X_{\eta'}$ such that $Z'_{\bar\eta}=Z$. Now if we write $U':=X_{\eta'}\setminus Z'$, then also $U'_{\bar\eta}=U$.

Write $Z_\eta$ for the image of $Z'$ under the finite map $X_{\eta'}\to X_\eta$ (give it the reduced structure, as $Z'$ is reduced, we have a map $Z'\to Z_\eta^{red}$ by the universal property). Then we see that $Z'\subseteq(Z_\eta)_{\eta'}$, so $\dim(Z')\leq\dim((Z_\eta)_{\eta'})=\dim(Z_\eta)$. But we also have $\dim(Z')\geq\dim(Z_\eta)$, so we conclude that $\dim(Z')=\dim(Z_\eta)$. We also know that $\dim(Z)=\dim(Z')$.

Let $\mathcal Z$ be the closure of $Z_\eta$ in $\mathcal X$. As we work over a smooth curve, the map $\mathcal Z\to B$ is flat (\cite[Corollary 4.3.10]{Liu}) and thus $\codim_{X_b}(Z_b)>j$ as well. Let $\mathcal U$ be the complement of $\mathcal Z$ in $\mathcal X$, giving a smooth family $\mathcal U\to\mathcal B$. This family then gives the specialization map $\Het^i(\mathcal U_{\bar\eta},\Lambda)\overset{cosp^\vee}{\to}\Het^i(\mathcal U_b,\Lambda)$. As $U_b:=\mathcal U_b$ is the complement of $Z_b$ in $X_b$, there is a natural map $\Het^i(U_b,\Lambda)\to\Het^i(F_jX_b,\Lambda)$. We also compute that $\mathcal U_{\bar\eta}=X_{\bar\eta}\setminus (Z_\eta)_{\bar\eta}\subseteq X_{\bar\eta}\setminus Z=U$, so we define the image of $\alpha$ to be the image of $\alpha_U$ under the composition $\Het^i(U,\Lambda)\to\Het^i(\mathcal U_{\bar\eta},\Lambda)\to\Het^i(U_b,\Lambda)\to\Het^i(F_jX_b,\Lambda)$.

We are left to check that this construction is compatible with restriction to opens. Suppose we have $F_jX_{\bar\eta}\subseteq U\subseteq W\subseteq X_{\bar\eta}$. Using the same construction as above, we now obtain an inclusion of families $\mathcal U\subseteq\mathcal W\to B$ and we have to show that the following diagram commutes
\[
\begin{tikzcd}
\mrH_c^\ast(U_b,\Lambda)\ar[r,"cosp"]\ar[d]&\mrH_c^\ast(\mathcal U_{\bar\eta},\Lambda)\ar[d]\\
\mrH_c^\ast(W_b,\Lambda)\ar[r,"cosp"]&\mrH_c^\ast(\mathcal W_{\bar\eta},\Lambda)
\end{tikzcd},
\]
where the vertical maps are the natural maps given by $j_!j^\ast\to \id$ for $j\colon U\hookrightarrow W$ the inclusion of an open. As the Poincaré dual of this map is the restriction to an open on cohomology, the commuativity indeed shows the claim after applying the restriction morphisms $\Het^i(U,\Lambda)\to \Het^i(\mathcal U_{\bar\eta},\Lambda)$ and $\Het^i(W,\Lambda)\to \Het^i(\mathcal W_{\bar\eta},\Lambda)$.

Let $j\colon\mathcal U\to \mathcal W$ be the inclusion and $f_{\mathcal U}$ and $f_{\mathcal W}$ be the maps to $B$ so that $f_{\mathcal W}\circ j=f_{\mathcal U}$. Then the vertical maps are given by taking the stalks at $b$ and $\bar\eta$ of the natural morphism $R(f_{\mathcal U})_!\Lambda=R(f_{\mathcal W})_!j_!j^\ast\Lambda\to R(f_{\mathcal W})_!\Lambda$. Using the functoriality of the cospecialization map \Cref{functorial-cosp}, we can apply it on this morphism, which precisely gives the commutativity of the diagram above.
\end{proof}

\subsection{Relation with the Étale Fundamental Group}
Here we note that we could have obtained \Cref{cohom-class-Enriques} without having to compute the full cohomology of an Enriques surfaces. After all, \cite[Table below Theorem 1.4.10]{Enr-Surf} gives that $\pi_1^{\e t}(S)\cong\Z/2\Z$. This implies that $\Het^1(S,\mu_{2^r})\cong\Hom_{\Z/2^r\Z}(\pi^{\e t}_1(S),\Z/2^r\Z)\cong\Z/2\Z$ and thus that $\Het^3(S,\mu_{2^r})\cong\Z/2\Z$. Now the Bockstein long exact sequence corresponding to $0\to\mu_2\to\mu_4\to\mu_2\to0$ gives us
\[
\Het^2(S,\mu_4)\to\Het^2(S,\mu_2)\twoheadrightarrow\Het^3(S,\mu_2)\overset{0}{\to}\Het^3(S,\mu_4)\isomto\Het^3(S,\mu_2)\overset{0}{\to}\Het^4(S,\mu_2)\hookrightarrow\Het^4(S,\mu_4),
\]
implying that $\Het^2(S,\mu_4)\to\Het^2(S,\mu_2)$ cannot be surjective.

The above argument only uses that $\pi_1^{\e t}(S)$ has some torsion class. Combining this with \Cref{cor-gen-IHC}, we obtain the following.
\begin{cor}
Let $S$ be a surface over an algebraically closed field $k$ with $\mathrm{char}(k)\neq\ell$. If $\pi_1^{\e t}(S)[\ell]\neq 0$, then the integral Tate conjecture fails for $X\times S$, where $X$ is the geometric generic fibre of a Lefschetz pencil as in \Cref{general statement}.\qed
\end{cor}

\bibliographystyle{abbrv}
\bibliography{refs}

\end{document}